\documentclass{cmslatex}
\usepackage[paperwidth=7in, paperheight=10in, margin=.875in]{geometry}
 \usepackage[backref,colorlinks,linkcolor=red,anchorcolor=green,citecolor=blue]{hyperref}
\usepackage{amsfonts,amssymb}
\usepackage{amsmath}
\usepackage{graphicx}
\usepackage{cite}
\usepackage{enumerate}
\usepackage{color,float,verbatim,mdframed,bm,booktabs,algorithm,algorithmic,subfig}
\usepackage{tabularx,mathrsfs,siunitx}
\usepackage{multirow}
\usepackage{dutchcal}
\usepackage{caption}
\usepackage{dsfont}
\usepackage{lscape}

\renewcommand{\theequation}{\arabic{section}.\arabic{equation}}

\newtheorem{Remark}{Remark}[section]
\newtheorem{Assumption}{Assumption}[section]

\DeclareMathOperator{\sgn}{sgn}

\newcommand{\calH}{\mathcal{H}}

\newcommand{\calE}{\mathcal{E}}
\newcommand{\calF}{\mathcal{F}}
\newcommand{\calJ}{\mathcal{J}}

\newcommand{\set}[1]{\left\{#1\right\}}
\newcommand{\abs}[1]{\left\vert#1\right\vert}

\newcommand{\RR}{\mathbb{R}}

\newcommand{\slg}{g}

   \allowdisplaybreaks
\begin{document}
 \title{A fast iterative thresholding and support-and-scale shrinking algorithm (FITS$^3$) for non-Lipschitz group sparse optimization (I): the case of least-squares fidelity\thanks{Received date, and accepted date (The correct dates will be entered by the editor).}}


          \author{Yanan Zhao\thanks{School of Mathematical Sciences, Nankai University and Nanjing Research Institute of Electronic Engineering  (mathzyn@mail.nankai.edu.cn). }
  	\and
  Qiaoli Dong\thanks{College of Science, Civil Aviation University of China (dongql@lsec.cc.ac.cn).}
  \and
  Yufei Zhao\thanks{School of Mathematical Sciences, Nankai University and LPMC
  	(matzyf@nankai.edu.cn).}
  \and
  	Chunlin Wu\thanks{Corresponding author, School of Mathematical Sciences, Nankai University, LPMC and NITFID (wucl@nankai.edu.cn).}
}
         \pagestyle{myheadings} \markboth{FITS$^3$ for group sparse recovery}{Yanan Zhao, Qiaoli Dong, Yufei Zhao and Chunlin Wu} \maketitle

          \begin{abstract}
               We consider to design a new efficient and easy-to-implement algorithm to solve a general group sparse optimization model with a class of non-convex non-Lipschitz regularizations, named as fast iterative thresholding and support-and-scale shrinking algorithm  (FITS$^3$). In this paper we focus on the case of a least-squares fidelity. FITS$^3$ is designed from a lower bound theory of such models and by integrating thresholding operation, linearization and extrapolation techniques. The  FITS$^3$   has two advantages. Firstly, it is quite efficient and  especially suitable for large-scale problems, because it adopts support-and-scale shrinking and does not need to solve any linear or nonlinear system. For two important special cases, the FITS$^3$ contains only simple calculations like matrix-vector multiplication and soft thresholding. Secondly, the FITS$^3$ algorithm  has a sequence convergence guarantee under proper assumptions. The numerical experiments and comparisons to recent existing non-Lipschitz group recovery algorithms demonstrate that, the proposed FITS$^3$  achieves  similar  recovery accuracies, but costs only around a half of the CPU time by the second fastest compared algorithm for median- or large-scale problems. 
          \end{abstract}

          \section{Introduction}\label{intro}
In this paper, we consider the following non-convex and non-Lipschitz group sparse optimization model with a least-squares fidelity
\begin{equation}\label{eq:object-func_00}
	\min_{x\in \mathbb{R}^{n}}\left\{\calE(x) :=  \frac{1}{2}\|A x-b\|_{2}^{2}+\alpha\sum_{\slg_i \in \mathcal{G}}\psi(\|x_{\slg_i}\|_p)\right\},
\end{equation}
where $A\in \mathbb{R}^{m\times n} (m<n)$ is a nonzero measurement matrix, $b\in \mathbb{R}^{m}$ is the observation data, $\alpha>0$ is the regularization parameter, $\psi(t): [0, +\infty) \rightarrow [0, +\infty)$ is a regularization function, $p\in[1,+\infty)$, $\mathcal{G}$ is the group index set, $x=(x_{\slg_1}, x_{\slg_2},..., x_{\slg_r})^\top\in\mathbb{R}^{n}$ with $x_{\slg_i}\ (\slg_i\in \mathcal{G})$ as its $i$th group. 
If each $\slg_i$ has only one element, the group sparse reconstruction model \eqref{eq:object-func_00} degenerates to the conventional non-group reconstruction model.
The regularization function $\psi$ considered in this paper is assumed to be concave and non-Lipschitz around 0; see Assumption~\ref{ass_psi} for more details.
One example of $\psi$ is  $\psi(t)=t^q$ ($0<q<1$).
Then the regularization term becomes the widely used non-convex and non-Lipschitz $\ell_{p,q}$ regularization as follows
$$
\|x\|_{p,q}^q=\sum_{\slg_i \in \mathcal{G} }\|x_{\slg_i}\|_p^q, \qquad
\forall x\in \RR^n.
$$
With such a regularization term, \eqref{eq:object-func_00} becomes the $\ell_2$-$\ell_{p,q}$ minimization model
\begin{equation}\label{eq:object-func_lpq}
	\min_{x\in \mathbb{R}^{n}} \frac{1}{2}\|A x-b\|_{2}^{2}+\alpha\sum_{\slg_i \in \mathcal{G}}\|x_{\slg_i}\|_p^q.
\end{equation}

The minimization model \eqref{eq:object-func_lpq} originates from the group sparse signal reconstruction problem. This problem actually is to find an $x\in \mathbb{R}^{n}$ to approximate the unknown ground
truth $\underline{x}\in \mathbb{R}^{n}$ from the following measurement
\begin{equation}\label{eqn:linear}
	b = A \underline{x} + \varepsilon,
\end{equation}
where $\varepsilon\in \mathbb{R}^{m}$ is the measurement noise with $m<n$.
The above linear system is an underdetermined linear system, so it usually has an infinite number of solutions.
If the ground truth $\underline{x}$ is known to be sparse or groupwise sparse, the regularization term based on the sparse prior information is used to constrain the solution space. Researchers first considered the problem of sparse signal recovery and proposed
the  $\ell_0$ ``norm" regularization model~(\cite{natarajan1995sparse,ompalg,blumensath2008Iterative,Song2020zero}), $\ell_1$ norm regularization model~(\cite{compressedsensing,bp,Candes2006stable,lassomodel,Yin2008}), and some other non-convex regularization models like $\ell_q\ (0<q<1)$ ``norm" regularization model~(\cite{Candes2008Enhancing,Chartrand2008Iteratively,Chen2010Lower,Lai2013Improved,Xu2012L12,
	nikolova2005analysis,Zhao2022,Figueiredo2007Majorization}),   $\ell_{1-2}$ model \cite{Lou2015} and  $L_1-\beta L_q$ ($(\beta, q)\in[0, 1]\times[1,\infty) \backslash(1, 1)$) model \cite{Ng2023}. Later on, how to enhance the recoverability of the groupwise sparse signals or images has attracted attentions in many scenarios like color imaging~(\cite{Compressed_se}), DNA microarrays~(\cite{Empir,recover_2008}), dynamic MRI ~(\cite{group_sparse}), source localization~(\cite{parse_signa}) and joint sparse recovery~(\cite{Adcock2019,Baron2009,Cotter2005}). For group sparse recovery, researchers proposed the $\ell_{2,1}$  regularization model \cite{yuan2006model}, the $\ell_{p,q}(p\geq 1,0\leq q\leq 1)$  regularization model \cite{hu2017group,rakotomamonjy2011ell,hu2024} and non-convex Lipschitz regularization models \cite{Lou2021,Lou2024,SCAD2007,MCP2015}. \eqref{eq:object-func_lpq} with  $\ell_{2,1}$ regularization, which is also called the group lasso model, is the first model
that encodes group information in addition to sparsity. The group lasso and other $\ell_{2,1}$ regularization forms like equality-constrained  group sparse  recovery  were designed for grouped variable selection in statistics, and 
%
%
they are convex problems with various efficient algorithms (\cite{Adcock2019,deng2013group,qin2013efficient,wright2009sparse}). However, the solution obtained from the convex $\ell_1$ or $\ell_{2,1}$ regularization models is not very precise (with magnitude reduction) and less sparse than expected (\cite{Lai2013Improved,Candes2008Enhancing,Xu2012L12,rakotomamonjy2011ell,hu2017group}). Instead, as shown in the literatures, 
the non-convex models with  either Lipschitz or non-Lipschitz  regularizations, usually show better  sparse reconstruction performances. 

Among the general non-convex non-Lipschitz group reconstruction model, the $\ell_{p,q}$ model \eqref{eq:object-func_lpq} seems to be the most widely studied.
The theoretical analysis of~\eqref{eq:object-func_lpq}, such as recovery bound and lower bound theories, has been developed in the literatures~\cite{feng2020l2,hu2017group}.
The recovery bound theory is to  estimate the distance between the minimizers
of the regularization model and the ground truth $\underline{x}$.
For the $\ell_{p,q}$  regularization model~\eqref{eq:object-func_lpq} with $0<q\leq 1\leq p\leq 2$, a global recovery bound for any point in a level set of the objective function was established in \cite{hu2017group} under the group restricted eigenvalue condition on $A$; also, under some assumptions like the activeness condition, a local recovery bound for local minimizers was derived in \cite{hu2017group} for the case $0<q<1\leq p$.
In~\cite{feng2020l2}, both the global and local recovery bounds were provided using the
group restricted isometry property assumption on $A$ for the $\ell_{2,q}$ ($0<q<1$) regularization model. The lower bound theory means the existence of a positive lower bound of nonzero entries of the local minimizers of a model; see, e.g., \cite{nikolova2005analysis,Chen2010Lower,article11,feng2020l2}. For the $\ell_{p,q}$ model with $p\geq 1$ and $0<q<1$, \cite{feng2020l2} provided a uniform lower bound theory for nonzero groups of local minimizers, which is a generalization of the remarkable lower bound theory in~\cite{Chen2010Lower} for the conventional $\ell_q$ regularization model.


However, how to design an efficient algorithm and establish the convergence analysis for the non-convex and non-Lipschitz regularization model in \eqref{eq:object-func_00}, even that with $\ell_{p,q}$ in \eqref{eq:object-func_lpq}, is not a trivial task.
%
%
There are some existing algorithms for solving \eqref{eq:object-func_lpq} so far. 
The proximal gradient method for group sparse optimization (PGM-GSO) algorithm was proposed in \cite{hu2017group} with a full and interesting analysis  for the $\ell_2$-$\ell_{p,q}$ minimization model, which is a generalization of the classical proximal forward-backward splitting method \cite{Combettes2005}. At each iteration,  for $p=1,2$ and $q=0,1/2, 2/3,1$, the subproblem  has an analytical solution; otherwise, a numerical algorithm such as the Newton method is used to calculate an approximate solution of the subproblem.
The global convergence of PGM-GSO was established, and its local linear convergence rate  was  derived under certain assumptions for the case $p=1$ and $0<q<1$.
The inexact iterative support shrinking algorithm with
proximal linearization for group sparse optimization (InISSAPL-GSO) was proposed in \cite{xue2019efficient} 
for solving the general $\ell_r$-$\ell_{p,q}$ minimization problem where $1\leq r\leq +\infty,\  p\geq 1$ and $0<q<1$.
It was designed by developing a group support inclusion property, like some previous works. InISSAPL-GSO has a two-loop structure with a scaled alternating direction method of multipliers (ADMM) as the inner solver. In \cite{feng2020l2}, inspired by the obtained lower bound theory,
the iterative reweighted least squares with thresholding algorithm (IRLS-th), which utilized a group support shrinking strategy similar to \cite{xue2019efficient}, was designed for the $\ell_2$-$\ell_{2,q}$ minimization problem. This IRLS-th algorithm was proposed on the basis of the IRLS algorithm (\cite{Chartrand2008Iteratively,Daubechies2010Iteratively,Lai2011Unconstrained,Lai2013Improved}), and  has an improved global convergence guarantee. We will compare these algorithms in Section \ref{sect_num}.
%

In this paper,  we present a simple and efficient algorithm with convergence guarantee to solve ~\eqref{eq:object-func_00}.
We first give a lower bound theory for the minimization model \eqref{eq:object-func_00}, which is a straightforward generalization of the conclusion in~\cite{feng2020l2}.
This naturally motivates a group support-and-scale shrinkage strategy 
in the iterative algorithm design.
Combining this strategy with a linearization method and an  extrapolation technique, we construct an iterative algorithm named as the fast iterative thresholding with support-and-scale shrinking algorithm (FITS$^3$). 
FITS$^3$ is a unified iterative algorithm for solving model~\eqref{eq:object-func_00} with $p\geq 1$, and in particular, it can be applied to solve the $\ell_{p,q}$ regularization model~\eqref{eq:object-func_lpq}.
The subproblem in each iteration is strongly convex, and the analytical solution of the subproblem is derived when $p=1$ and $p=2$.
Furthermore, we prove that the sequence generated by FITS$^3$ with an inexact inner solver converges globally to a stationary point of the objective function in \eqref{eq:object-func_00}. The convergence rate of FITS$^3$  can also be shown using a similar analysis as in the literatures. The numerical experiments show that, with comparable reconstruction accuracy, the FITS$^3$ algorithm has a definite advantage in terms of running time over other algorithms, especially for median- or large-scale problems.

The remainder of this paper is organized as follows. In Section \ref{sect2_nota}, we give some notations and preliminaries. In Section \ref{sec:algo}, we propose our new algorithm named as FITS$^3$. The convergence analysis of FITS$^3$ is established in Section \ref{sec:con-ana}. The numerical experiments and comparisons are presented in Section \ref{sect_num}.

\section{Notation and preliminary}\label{sect2_nota}

Suppose that $x=(x_1, x_2,...,x_n)^\top$ is an $n$-dimensional vector.
We use $[n]=\set{1,2,...,n}$ to denote the index set of $x$ and $x_j$ is the $j$th element of $x$. Let $\mathcal{G}=\set{\slg_1,...,\slg_i,...,\slg_r}$ denote a collection of index sets of groups with $\slg_i=\set{n_{i-1}+1,...,n_{i-1}+n_i}$, $n_0=0$, $n_i\in[n]$ ($1\leq i\leq r$), and $\sum_{i=1}^r n_i=n$.
With this,
$x=(x_{\slg_1}^\top, x_{\slg_2}^\top,..., x_{\slg_r}^\top)^\top$ represents a group structure of $x$ as follows 
$$
x=(\underbrace{x_1,...,x_{n_1}}_{x_{\slg_1}^\top},...,\underbrace{x_{n_{i-1}+1},...,x_{n_{i-1}+n_i}}_{x_{\slg_i}^\top},...,\underbrace{x_{n-n_r+1},...,x_n}_{x_{\slg_r}^\top})^{\top},
$$
where $x_{\slg_i}$ is the $i$th group of $x$.
Let $\#\slg_i$ represent the cardinality of $\slg_i$. We have $\sum_{i=1}^{r}(\#\slg_i)=n$.
Note that we only consider the case with no overlapping between groups.
For a group $x_{\slg_i}$, $x_{\slg_i}=0$ means that $x_j=0, \forall j\in \slg_i$. 
See, e.g., \cite{feng2020l2} for details.
For a matrix $A\in \mathbb{R}^{m\times n}$, we also partition it into submatrices associating to the group structure of $x$, i.e.,
$$
A=\left[
\begin{array}{cccc}
	A_{1,\slg_1} & A_{1,\slg_2} & ... & A_{1,\slg_r} \\
	... & ... & ... & ... \\
	A_{m,\slg_1} & A_{m,\slg_2} & ... & A_{m,\slg_r} \\
\end{array}
\right].
$$
Let $\mathcal{S}$ be a subset of $\mathcal{G}$, then $x_\mathcal{S}$ denotes the sub-vector of $x$ consisting of the elements determined by indices collected in $\mathcal{S}$. Similarly, $A_\mathcal{S}$ denotes the sub-matrix of $A$ consisting of the columns determined by indices collected in $\mathcal{S}$.
We also recall the following notations for support sets:
$$
\mathcal{S}(x)=\left\{ \slg_i \in \mathcal{G}: \|x_{\slg_i}\|_p \neq 0 \right\},\qquad
s(x_{\slg_i})=\left\{j\in \slg_i: x_j\neq 0 \right\},
$$
where 	$\mathcal{S}(x)$ is called the group support set of $x$ and $s(x_{\slg_i})$ is called the support set of $x_{\slg_i}$ (see, e.g., \cite{feng2020l2,xue2019efficient}).

We assume the regularization function $\psi$ satisfying the following
\begin{Assumption}\label{ass_psi}
	\begin{enumerate}[(i)]
		\item The function $\psi: [0, +\infty) \rightarrow [0, +\infty)$ is a concave coercive~(\cite[Definition 3.25]{rockafellar2009variational}) function, and  is $C^2$ on $(0, +\infty)$ with $\psi(0)=0$.
		\item  $\psi^{\prime}(0+)=+\infty$.
		\item $\psi^{\prime \prime}$ is an increasing function on $(0, +\infty)$.
	\end{enumerate}	
\end{Assumption}
Assumption \ref{ass_psi}$(i)$ implies several consequences. Firstly, $\psi^{\prime}(t)|_{(0, +\infty)}> 0$.  Secondly, the objective function $\calE$ in \eqref{eq:object-func_00} is coercive, and thus  has always a minimizer. Thirdly, $0$ is the strict minimizer of  $\psi$, and thus \eqref{eq:object-func_00} has a trivial solution $0$ if $b=0$. Therefore, we consider the case of $b\neq0$ in this paper. Assumption \ref{ass_psi}$(ii)$ implies that $\psi$ is  non-Lipschitz.  Herein,  we list some functions that satisfy Assumption \ref{ass_psi}: $\psi_1(t)=t^q, 0<q<1; \psi_2(t)=\log(t^q+1), 0<q<1$ (\cite{nikolova2005analysis,zeng2019iterative}). 
\begin{Remark}\label{property_psi}
	We list some straightforward properties (see, e.g., \cite{xue2019efficient}) of $\psi$.
	\begin{enumerate}[(i)]
		\item For any $c > 0$, $\psi^{\prime}(t)$ is $ L_{c}$-Lipschitz continuous on $[c,+\infty)$, i.e., there exists  constant $ L_{c} > 0$ determined by $c$, such that $
		\abs{\psi^{\prime}(t) - \psi^{\prime}(s)} \leq L_{c} \abs{t-s},
		$ for any $t,s \in [c, +\infty)$.
		
		\item For any $w\in\mathbb{R}^{d}$ and $p\geq 1$, the subdifferential of $\psi(\|w\|_p)$ is
		\begin{equation}\label{eq_subdiff}
			\partial \psi(\|w\|_p)=\left\{
			\begin{aligned}
				&\prod_{j=1}^{d}(-\infty,+\infty), \; \;\text{if}\;\; \|w\|_p=0;\\
				&\prod_{j=1}^{d} \mu_{j},\qquad\quad\;\;\; \;\text{if}\;\;\|w\|_p\neq 0,
			\end{aligned}
			\right.
		\end{equation}
		where $\prod$ means the Cartesian product of sets, and
		\begin{equation*}
			\mu_{j}=\left\{
			\begin{aligned}
				&\set{\psi^{\prime}(\|w\|_p)\|w\|_p^{1-p}|w_j|^{p-1}\sgn(w_j)}, \quad p>1;\\
				&\set{\psi^{\prime}(\|w\|_1) \sgn(w_j)},\qquad\quad\qquad\; \qquad p=1, w_j \neq 0;\\
				&[-\psi^{\prime}(\|w\|_1), \; \psi^{\prime}(\|w\|_1)],\;\, \;\qquad \quad\;\;\qquad p=1, w_j=0;
			\end{aligned}
			\right.
		\end{equation*}
		with the sign function  $\sgn(\cdot)$.
	\end{enumerate}
\end{Remark}

Clearly, the subdifferential of $\calE$ in \eqref{eq:object-func_00} is
$
\partial \calE(x) 
= A^{\top}(Ax-b)
+\alpha \prod_{\slg_i \in \mathcal{G}}\partial\psi(\|x_{\slg_i}\|_p),
$
where $\partial\psi(\|x_{\slg_i}\|_2)$ is calculated by \eqref{eq_subdiff}.


Next, we give two inequalities that will be frequently used in the convergence analysis.
\begin{lemma}(\cite{hu2017group})\label{lemma_inequa}
	For any $w\in\mathbb{R}^{d}$, $0<\gamma_1\leq \gamma_2$, we have the following inequality
	\begin{equation}\label{eq_ineq1}
		\|w\|_{\gamma_2}\leq \|w\|_{\gamma_1}.
	\end{equation}
\end{lemma}

\begin{lemma}(\cite{xue2019efficient})\label{lemma_ineq2}
	For any $w\in\mathbb{R}^{d}$, $\gamma>0$, then there exists $\widetilde{C}_{\gamma}>0$ such that
	\begin{equation}\label{eq_ineq3}
		\|w\|_{\gamma}\leq \widetilde{C}_{\gamma} \|w\|_{\gamma+1},
	\end{equation}
	particularly, for $\gamma=1$ it holds that $\|w\|_{1}\leq \sqrt{d} \|w\|_{2}$.
\end{lemma}

\section{Algorithm}
\label{sec:algo}
In this section, we start from a theorem on lower bound theory, and combine thresholding, linearization method and the  extrapolation technique, to develop a new efficient and easy-to-implement algorithm, i.e., the fast iterative thresholding with support-and-scale shrinking algorithm (FITS$^3$) for the non-convex non-Lipschitz minimization problem in \eqref{eq:object-func_00}.

\begin{theorem}\label{theorem_lowerbound}
	(Lower bound theory) If $\calE$ in \eqref{eq:object-func_00} has a local minimum at $x^{*}$, then for any $\slg_i \in \mathcal{G}$,
	$$
	x^{*}_{\slg_i} \neq 0\; \Longrightarrow \;\|x^{*}_{\slg_i}\|_p\geq\kappa\text{ with }
	\kappa=\left\{
	\begin{aligned}
		& (\psi^{\prime\prime})^{-1}\left(\frac{-\|A\|_2^2}{\alpha}\right), \;\quad\;\;\; 1\leq p\leq 2,\\
		& (\psi^{\prime\prime})^{-1}\left(\frac{-\|A\|_2^2}{\alpha (\#\slg_i)^{1-2/p}}\right), \; p>2.
	\end{aligned}
	\right.
	$$
\end{theorem}
\begin{proof}
	This proof is a straightforward generalization of \cite[Theorem 3.1]{feng2020l2}.
\end{proof}

Theorem \ref{theorem_lowerbound} implies a group support inclusion relationship.
Suppose that a local minimizer $x^\ast$ is very near to some iteration point $x^k$ such that $\|x^\ast-x^k\|_p<\kappa$. Then for $\slg_i\in\mathcal{G}$ with $x^{k}_{\slg_i}= 0$, $\|x^\ast_{\slg_i}\|_p\leq \|x^\ast_{\slg_i}-x^k_{\slg_i}\|_p+\|x^{k}_{\slg_i} \|_p<\kappa$, which implies $x^\ast_{\slg_i}=0$ by Theorem~\ref{theorem_lowerbound}.
Therefore, $\mathcal{S}(x^\ast)\subseteq\mathcal{S}(x^k):=\mathcal{S}^{k}$. Such
support inclusion relationship naturally motivates a group support shrinking strategy at each iteration, like \cite{Zhao2022,xue2019efficient,feng2020l2,article11} and references therein.
In addition, considering the finite word length of real computers and to avoid extremely large linearization weights described later,
we relax the computation of the group support set $\mathcal{S}^k$ by using a thresholding operation like \cite{feng2020l2,article11,Wen2010}. In particular, let $\tilde{x}^{k}$ be as follows
$$
\tilde{x}^{k}_{\slg_i}=\left\{
\begin{aligned}
	&x^{k}_{\slg_i},\;\;\; \text{if}\;\,\|x^{k}_{\slg_i}\|_p\geq \tau;\\
	&0,\quad\;\; \text{if}\;\,\|x^{k}_{\slg_i}\|_p<\tau,
\end{aligned}
\right.
$$
for some thresholding parameter $\tau$, and denote $\tilde{\mathcal{S}}^{k}=\mathcal{S}(\tilde{x}^k)$. Obviously, it holds that $\tilde{\mathcal{S}}^{k}\subseteq \mathcal{S}^{k}$ (\cite{article11}).
Consequently, given $x^k$, we can compute $x^{k+1}$ by 
\begin{equation}\label{eq_minxin}
	\begin{aligned}
		&\min_{x\in \mathbb{R}^{n}} \set{\frac{1}{2}\|A x-b\|_{2}^{2}+\alpha\sum_{\slg_i \in \tilde{\mathcal{S}}^{k}}\psi(\|x_{\slg_i}\|_p)}\\
		&\text{s.t.}\;x_{\slg_i}=0, \; \forall \slg_i \in \mathcal{G}\setminus \tilde{\mathcal{S}}^{k}.
	\end{aligned}
\end{equation}

Note that the problem \eqref{eq_minxin} is still non-convex and  cannot be solved directly. Here, we linearize both the fidelity term and the regularization term \cite{Figueiredo2007Majorization,Zhao2022} at $\tilde{x}^{k}$, and $x^{k+1}$ can be computed by
\begin{equation}\label{eq_minxinr}
	\begin{aligned}
		&\min_{x\in \mathbb{R}^{n}} \set{
			\begin{aligned}
				\frac{1}{2}\|A\tilde{x}^{k}-b\|_{2}^{2}+&(x-\tilde{x}^{k})^\top A^\top(A\tilde{x}^{k}-b) +\frac{\beta}{2}\|x-\tilde{x}^{k}\|^{2}_{2}+\\
				&\alpha\sum_{\slg_i \in \tilde{\mathcal{S}}^{k}}\left[\psi(\|\tilde{x}_{\slg_i}^{k}\|_p)
				+\psi^{\prime}(\|\tilde{x}_{\slg_i}^{k}\|_p)(\|x_{\slg_i}\|_p-\|\tilde{x}^{k}_{\slg_i}\|_p)\right]
			\end{aligned}
		}\\
		&\text{s.t.}\;x_{\slg_i}=0, \; \forall \slg_i \in \mathcal{G}\setminus \tilde{\mathcal{S}}^{k},
	\end{aligned}
\end{equation}
where $\beta>\|A^{\top}A\|_2$. For \eqref{eq_minxinr}, we can simply substitute the constraint into the objective function, which not only eliminates constraints but also reduces the scale of the problem. Let $B^k=A_{\tilde{\mathcal{S}}^{k}}$. Clearly, one has $A\tilde{x}^{k}=B^k\tilde{x}^{k}_{\tilde{\mathcal{S}}^{k}}$.
When the group support set $\mathcal{S}(x)\subseteq \tilde{\mathcal{S}}^{k}$, we  derive that
\begin{equation*}
	\begin{split}
		&\frac{1}{2}\|A\tilde{x}^{k}-b\|_{2}^{2}+(x-\tilde{x}^{k})^\top A^\top(A\tilde{x}^{k}-b) +\frac{\beta}{2}\|x-\tilde{x}^{k}\|^{2}_{2}\\
		=&\frac{1}{2}\|B^k\tilde{x}^{k}_{\tilde{\mathcal{S}}^{k}}-b\|_{2}^{2}+ (x_{\tilde{\mathcal{S}}^{k}}-\tilde{x}^{k}_{\tilde{\mathcal{S}}^{k}})^\top (B^k)^\top (B^k\tilde{x}^{k}_{\tilde{\mathcal{S}}^{k}}-b) +\frac{\beta}{2}\|x_{\tilde{\mathcal{S}}^{k}}-\tilde{x}^{k}_{\tilde{\mathcal{S}}^{k}}\|^{2}_{2}.
	\end{split}
\end{equation*}
Therefore, the problem in \eqref{eq_minxinr} is equivalent to the following unconstrained minimization problem
\begin{equation}\label{eq_alix}
	\left\{
	\begin{aligned}
		&x^{k+1}_{\tilde{\mathcal{S}}^{k}}=\arg\min_{x_{\tilde{\mathcal{S}}^{k}}} \set{
			\begin{aligned}
				\frac{1}{2}\|B^k\tilde{x}^{k}_{\tilde{\mathcal{S}}^{k}}-b\|_{2}^{2}+& (x_{\tilde{\mathcal{S}}^{k}}-\tilde{x}^{k}_{\tilde{\mathcal{S}}^{k}})^\top (B^k)^\top (B^k\tilde{x}^{k}_{\tilde{\mathcal{S}}^{k}}-b) +\frac{\beta}{2}\|x_{\tilde{\mathcal{S}}^{k}}-\tilde{x}^{k}_{\tilde{\mathcal{S}}^{k}}\|^{2}_{2}+\\
				&\alpha\sum_{\slg_i \in \tilde{\mathcal{S}}^{k}}\left[\psi(\|\tilde{x}_{\slg_i}^{k}\|_p)
				+\psi^{\prime}(\|\tilde{x}_{\slg_i}^{k}\|_p)(\|x_{\slg_i}\|_p-\|\tilde{x}^{k}_{\slg_i}\|_p)\right]
			\end{aligned}
		}\\
		&x^{k+1}_{\mathcal{G}\setminus \tilde{\mathcal{S}}^{k}}=0.
	\end{aligned}
	\right.
\end{equation}

The problem  \eqref{eq_alix} can be solved efficiently. To further speed up the computation, we borrow the idea of the  extrapolation technique  (\cite{Nesterov1983A,Beck2009A,Zhao2022,Ochs2014iPiano}),
and introduce a new variable $z^k$ defined as:
\begin{equation}\label{eq_zkk}
	\left\{
	\begin{aligned} &z^{k}_{\tilde{\mathcal{S}}^{k}}=\tilde{x}^{k}_{\tilde{\mathcal{S}}^{k}}+t_{k}(\tilde{x}^{k}_{\tilde{\mathcal{S}}^{k}}-\tilde{x}^{k-1}_{\tilde{\mathcal{S}}^{k}});\\
		&z^{k}_{\mathcal{G}\setminus\tilde{\mathcal{S}}^{k}}=0,
	\end{aligned}
	\right.
\end{equation}
where $t_k$ is the parameter of extrapolation. Note that $\mathcal{S}(z^k)\subseteq\tilde{\mathcal{S}}^{k}$.
We then propose to update $x^{k+1}$ as follows
\begin{equation}\label{eq_aliz}
	\left\{
	\begin{aligned}
		&x^{k+1}_{\tilde{\mathcal{S}}^{k}}=\arg\min_{x_{\tilde{\mathcal{S}}^{k}}} \set{
			\begin{aligned}
				\frac{1}{2}\|B^kz^{k}_{\tilde{\mathcal{S}}^{k}}-b\|_{2}^{2}+& (x_{\tilde{\mathcal{S}}^{k}}-z^{k}_{\tilde{\mathcal{S}}^{k}})^\top (B^k)^\top (B^kz^{k}_{\tilde{\mathcal{S}}^{k}}-b) +\frac{\beta}{2}\|x_{\tilde{\mathcal{S}}^{k}}-z^{k}_{\tilde{\mathcal{S}}^{k}}\|^{2}_{2}+\\
				&\alpha\sum_{\slg_i \in \tilde{\mathcal{S}}^{k}}\left[\psi(\|\tilde{x}_{\slg_i}^{k}\|_p)
				+\psi^{\prime}(\|\tilde{x}_{\slg_i}^{k}\|_p)(\|x_{\slg_i}\|_p-\|\tilde{x}^{k}_{\slg_i}\|_p)\right]
			\end{aligned}
		}\\
		&x^{k+1}_{\mathcal{G}\setminus \tilde{\mathcal{S}}^{k}}=0.
	\end{aligned}
	\right.
\end{equation}

We now  define
$$
y^{k}=z^{k}_{\tilde{\mathcal{S}}^{k}}-\frac{1}{\beta} (B^k)^\top (B^kz^{k}_{\tilde{\mathcal{S}}^{k}}-b),
$$
and reformulate \eqref{eq_aliz} as follows
\begin{equation}\label{eq_leftali}
	x^{k+1}_{\slg_i}=\left\{
	\begin{aligned}
		&\text{arg}\min_{x_{\slg_i}}\left\{
		\frac{\beta}{2}\|x_{\slg_i}-y^{k}_{\slg_i}\|_2^{2}+\alpha\psi^{\prime}(\|\tilde{x}_{\slg_i}^{k}\|_p)\|x_{\slg_i}\|_p\right\}, \;\; \forall \slg_i  \in \tilde{\mathcal{S}}^{k} \\
		&0,  \;\; \forall \slg_i \in \mathcal{G} \setminus \tilde{\mathcal{S}}^{k},
	\end{aligned}
	\right.
\end{equation}
where $y^{k}_{\slg_i}$ is the component of $y^{k}$. The problem in \eqref{eq_leftali} is a strongly convex problem and there are many efficient methods to solve it, such as the Newton method.
For two most interesting cases with $p=1$ and $p=2$, this problem has analytic solutions. When $p=1$, its analytic solution is as follows~(\cite{denose_bysoft})
$$
x^{k+1}_{j}=\max\left\{ |y_{j}^{k}|-\alpha\psi^{\prime}(\|\tilde{x}_{\slg_i}^{k}\|_1)/\beta, 0\right\}\sgn(y_{j}^{k}), \; \forall j\in \slg_i,\, \slg_i  \in \tilde{\mathcal{S}}^{k}.
$$
When $p=2$, the analytic solution is the multi-dimensional shrinkage~(\cite{article_yinwotao})
$$
x^{k+1}_{\slg_i}=\max\left\{ \|y_{\slg_i}^{k}\|_2-\alpha\psi^{\prime}(\|\tilde{x}_{\slg_i}^{k}\|_2)/\beta, 0\right\}\frac{y_{\slg_i}^{k}}{\|y_{\slg_i}^{k}\|_2}, \; \forall \slg_i  \in \tilde{\mathcal{S}}^{k}.
$$
We mention that with $p=1$, the minimization problem~\eqref{eq:object-func_00} is suitable for intra-group sparse recovery (\cite{xue2019efficient,hu2017group}). When $p> 1$,  $p=2$ is the most typical choice in~\eqref{eq:object-func_00} (\cite{hu2017group,feng2020l2}).

In summary, we give the details of the  fast iterative thresholding with support and scale shrinking algorithm (FITS$^3$) in Algorithm \ref{alg_accelerated_ITS3}. Therein  $\bar t\in(0,1]$ is chosen so that  Algorithm \ref{alg_accelerated_ITS3} is convergent; see Section \ref{sec:con-ana} for details.
\begin{algorithm}[htbp]
	\caption{FITS$^3$: fast iterative thresholding with support-and-scale shrinking}
	\label{alg_accelerated_ITS3}
	\begin{algorithmic}
		\STATE{Input $\alpha, \beta,\mathrm{Tol},\tau, x^{0},x^{-1}=\tilde{x}^{-1}=x^{0}, k=0,\text{MAXit}$ and $\set{t_{k}}\subset[0,{\bar t})$ satisfying $\sup_kt_k<{\bar t}$.}
		\WHILE{$k\leq \text{MAXit}$ \; }
		\STATE{1. Thresholding:
			\begin{equation}\label{eq:threshold1}
				\nonumber
				\tilde{x}^{k}_{\slg_i}=\left\{
				\begin{aligned}
					&x^{k}_{\slg_i},\;\;\; \text{if}\;\,\|x^{k}_{\slg_i}\|_p\geq \tau;\\
					&0,\quad\;\; \text{if}\;\,\|x^{k}_{\slg_i}\|_p<\tau.
				\end{aligned}
				\right.
			\end{equation}
		}
		\STATE{2. Find $\tilde{\mathcal{S}}^{k}=\left\{ \slg_i \in \mathcal{G} : \|\tilde{x}^k_{\slg_i}\|_p \neq 0 \right\}$.}
		\STATE{3. Update $B^k$ and $z^{k}$:	
			\begin{equation}
				\nonumber
				B^k=A_{\tilde{\mathcal{S}}^{k}}, \qquad
				\left\{
				\begin{aligned} &z^{k}_{\tilde{\mathcal{S}}^{k}}=\tilde{x}^{k}_{\tilde{\mathcal{S}}^{k}}+t_{k}(\tilde{x}^{k}_{\tilde{\mathcal{S}}^{k}}-\tilde{x}^{k-1}_{\tilde{\mathcal{S}}^{k}});\\
					&z^{k}_{\mathcal{G}\setminus\tilde{\mathcal{S}}^{k}}=0.
				\end{aligned}
				\right.
		\end{equation}}
		\STATE{4. Compute $y^k$:
			$y^{k}=z^{k}_{\tilde{\mathcal{S}}^{k}}-\frac{1}{\beta} (B^k)^\top (B^kz^{k}_{\tilde{\mathcal{S}}^{k}}-b)$.}
		\STATE{5. Compute $x^{k+1}$ by solving \eqref{eq_leftali}.
		}
		\STATE{6. If $\| x^{k+1}-x^{k}\|_{2}/\|x^{k}\|_{2}<\mathrm{Tol}$,
			break.
			
		}
		\ENDWHILE
		\RETURN $x^{k+1}$
	\end{algorithmic}
\end{algorithm}

\begin{remark}
	Here we wish to point out some differences between our  FITS$^3$ algorithm and 
	several typical existing non-convex non-Lipschitz group sparse recovery methods \cite{hu2017group,xue2019efficient,feng2020l2}. Our FITS$^3$ algorithm integrates thresholding, support-and-scale shrinking, linearization and extrapolation techniques, and is applicable to more general non-Lipschitz regularization  than $\ell_{p,q}$ regularization.
	The PGM-GSO in \cite{hu2017group} and the InISSAPL-GSO in \cite{xue2019efficient} consider the $\ell_{p,q}$ regularization, and the IRLS-th in \cite{feng2020l2} is applicable only to $\ell_{2,q}$ regularization.
	All of these three algorithms  do not use extrapolation acceleration.
	The PGM-GSO in \cite{hu2017group} depends on the proximal operator of regularization term, which is not always explicit.
	The InISSAPL-GSO in \cite{xue2019efficient}   is with a two-loop algorithmic structure for any $p,q$.
	Our  FITS$^3$ algorithm and PGM-GSO do not need to solve any linear system, while both InISSAPL-GSO and IRLS-th  solve a linear system at each iteration.
	
\end{remark}

\begin{remark}
	We mention that, we  got aware very recently of  a closely related efficient and sophisticated two-phase method  with a shrinkage phase and a subspace optimization phase  proposed in \cite{Wen2010}. There are some differences between our FITS$^3$ and the algorithm in \cite{Wen2010}. First, the algorithm in \cite{Wen2010} was  designed  for $\ell_1$ minimization in
	non-group sparse reconstruction, while our FITS$^3$ is for  group sparse recovery with a general non-Lipschitz regularization. Second, their algorithm   uses continuation and efficient quadratic solvers at each subspace optimization step, while our FITS$^3$ integrates extrapolation acceleration and linearization  technique.
\end{remark}
\section{Convergence analysis}
\label{sec:con-ana}
In this section, we give the convergence analysis of  the  proposed FITS$^3$ algorithm.
Our proof relies on  K\L~property~\cite{Kurdyka1998gradients,Lojasiewicz1963Une} of a function (or K\L~function) which has extensive applications in 
analyzing non-convex optimization algorithms  \cite{Attouch2009convergence,Attouch2010Proximal,Attouch2013Convergence,Bolte2014Proximal,
	Zhao2022,Ochs2014iPiano,feng2020l2}. 

For convenience of description, we introduce the following auxiliary functions
\begin{equation*}
	\begin{split}
		\calF^k(x)
		& =
		\frac{1}{2} \| A z^k -b \|_{2}^{2}
		+(x-z^k)^\top A^\top(Az^k-b) + \frac{\beta}{2}\|x-z^k\|_{2}^{2} \\
		& \phantom{=;}  + \alpha \sum_{\slg_i  \in \tilde{\mathcal{S}}^{k}} \left[\psi(\|\tilde{x}_{\slg_i}^{k}\|_p)+\psi^{\prime}(\|\tilde{x}_{\slg_i}^{k}\|_p)(\|x_{\slg_i}\|_p-\|\tilde{x}_{\slg_i}^{k}\|_p)\right],\quad k=0,1,2,\cdots,
	\end{split}
\end{equation*}
which can be decomposed into $\calF^k(x)=\calJ_0^k(x)+\sum_{\slg_i  \in \tilde{\mathcal{S}}^{k}} \calJ_{\slg_i}^k(x_{\slg_i})$  with
\begin{equation*}
	\begin{split}
		\calJ_0^k(x)
		& =\frac{1}{2} \| A z^k -b \|_{2}^{2}+
		\sum_{\slg_i  \in \mathcal{G} \setminus \tilde{\mathcal{S}}^{k}} 
		(x_{\slg_i}-z^k_{\slg_i})^\top( A^\top(Az^k-b))_{\slg_i} + \frac{\beta}{2}\sum_{\slg_i \in \mathcal{G} \setminus \tilde{\mathcal{S}}^{k}}\|x_{\slg_i}-z^k_{\slg_i}\|_{2}^{2},\\
	\calJ_{\slg_i}^k(x_{\slg_i})
	& =
	(x_{\slg_i}-z^k_{\slg_i})^\top( A^\top(Az^k-b))_{\slg_i} + \frac{\beta}{2}\|x_{\slg_i}-z^k_{\slg_i}\|_{2}^{2} \\
	&\quad+ \alpha \left[\psi(\|\tilde{x}_{\slg_i}^{k}\|_p)+\psi^{\prime}(\|\tilde{x}_{\slg_i}^{k}\|_p)(\|x_{\slg_i}\|_p
	-\|\tilde{x}_{\slg_i}^{k}\|_p)\right].
\end{split}
\end{equation*}
Note that $x_{\slg_i}^{k+1}$, $\slg_i  \in \tilde{\mathcal{S}}^{k}$ in \eqref{eq_leftali} is essentially updated by solving
\begin{equation}\label{eq:f}
\begin{aligned}
	&\min_{x_{\slg_i}}\, \calJ_{\slg_i}^k(x_{\slg_i}).
\end{aligned}
\end{equation}
We will establish our convergence result under the following (inexact) condition that
\begin{equation}\label{eq_inexact_condition}
\lambda^{k+1}_{\slg_i}\in\partial \calJ_{\slg_i}^k(x^{k+1}_{\slg_i})\,\,\mbox{with}\,\,\|\lambda^{k+1}_{\slg_i}\|_2
\le\frac{\beta\varepsilon}2\|x^{k+1}_{\slg_i}-\tilde x^{k}_{\slg_i}\|_2, \quad \slg_i\in\tilde{\mathcal{S}}^{k}
\end{equation}	
exists for a given $\varepsilon\in [0,1)$, since   the subproblem \eqref{eq:f} needs an iterative solver for $p$ rather than 1 and 2. 
It is worthy to notice that \eqref{eq_inexact_condition} implies
\begin{equation}\label{eq_inexact_condition2}
\|\lambda^{k+1}_{\tilde{\mathcal{S}}^{k}}\|_2
\le\frac{\beta\varepsilon}2\|x^{k+1}_{\tilde{\mathcal{S}}^{k}}-\tilde x^{k}_{\tilde{\mathcal{S}}^{k}}\|_2,
\end{equation}
no matter how $\lambda^{k+1}_{\slg_i},$ $\slg_i\in \mathcal{G} \setminus \tilde{\mathcal{S}}^{k}$ are chosen. Because of the (inexact) condition, we let in Algorithm \ref{alg_accelerated_ITS3}
${\bar t}\overset\triangle={\bar t}(\varepsilon)=\frac{\sqrt{(\beta+\|A^\top A\|_2)^2-4\beta\varepsilon\|A^\top A\|_2}-(\beta-\|A^\top A\|_2)}{2\|A^\top A\|_2}\in(0,1]$ for the latter convergence analysis. 

\begin{remark}
If $p=1$ or $p=2$ in our 
FITS$^3$ algorithm, we can let $\varepsilon=0$ in \eqref{eq_inexact_condition} and  \eqref{eq_inexact_condition2}, and thus have ${\bar t}=1$, due to the closed form solutions given in Section \ref{sec:algo}.
\end{remark}

We first observe a finite convergence property of the group support sets. According to  Algorithm~\ref{alg_accelerated_ITS3},  $\tilde{\mathcal{S}}^{k+1}\subseteq \mathcal{S}^{k+1}\subseteq\tilde{\mathcal{S}}^{k}\subseteq \mathcal{S}^{k}$ holds, like \cite{article11} and references therein. Since $\mathcal{G}$ is a finite set and $\tilde{\mathcal{S}}^{k}\subseteq \mathcal{S}^{k}\subseteq\mathcal{G}$,  there exists an integer $K> 0$ such that
\begin{equation}\label{eq_decrease_mont}
\tilde{\mathcal{S}}^{k}\equiv \mathcal{S}^{k} \equiv  \mathcal{S}^{K},\;\forall k \geq K.
\end{equation}	
Subsequently, one has $\tilde{x}^{k} \equiv x^{k},\;\forall k \geq K$.

We now  present some lemmas where the following value function 
\begin{equation}\label{eq_ht}
\mathcal{H}(x,u)=\calE(x)+\frac{\beta(1-\varepsilon)}{2}\|x-u\|_{2}^{2},
\end{equation}
plays a key role.

\begin{lemma}\label{lem-bound-sufficient-decrease}
For any $\beta>\|A^{\top}A\|_2$ and the extrapolation parameters $\set{t_{k}}\subset[0,{\bar t})$ satisfying $\hat{t}=\sup_{k} t_k<{\bar t}$,
let $\set{x^{k}}$ be a sequence generated by FITS$^3$. Then, the following properties hold:
\begin{enumerate}[(i)]
	\item For any $k\geq K+1$, the sequence $\set{\mathcal{H}(x^{k},x^{k-1})}$ is monotonically nonincreasing. In particular, it holds that
{\small
		\begin{equation*}
		\mathcal{H}(x^{k},x^{k-1})-\mathcal{H}(x^{k+1},x^{k})\geq \left(-\frac{\|A^{\top}A\|_2}{2}\hat{t}^2-\frac{\beta-\|A^{\top}A\|_2}{2}\hat{t}
		+\frac{\beta(1-\varepsilon)}{2}\right)\|x^{k}-x^{k-1}\|_2^2.
	\end{equation*}
}
	\item $\lim_{k \to \infty} \| x^{k} - x^{k-1} \|_{2} = 0$ and the sequence $\set{x^k}$ is bounded. 
	\item $\zeta:=\lim_{k \to \infty}\calE(x^k)$ exists. Moreover, $\calE\equiv \zeta$ on $\Omega$, where $\Omega$ is the set of accumulation points of $\set{x^k}$.
	\item $\lim_{k \to \infty}\mathcal{H}(x^k,x^{k-1})=\zeta$. Moreover, $\calH\equiv \zeta$ on $\Upsilon$, where $\Upsilon:=\set{(x,x):x\in \Omega}$ is the set of accumulation points of $\set{(x^k,x^{k-1})}$.
\end{enumerate}
\end{lemma}

\begin{proof}
$(i)$ Letting $x=\tilde{x}^{k}$ in $\calF^k(x)$, we get
\begin{equation*}
	\calF^k(\tilde{x}^{k})=
	\frac{1}{2} \| Az^k - b \|_{2}^{2}
	+(\tilde{x}^{k}-z^k)^\top A^\top(Az^k-b) + \frac{\beta}{2}\|\tilde{x}^{k}-z^k\|_{2}^{2} + \alpha \sum_{\slg_i \in \tilde{\mathcal{S}}^{k}} \psi(\|\tilde{x}_{\slg_i}^{k}\|_p).
\end{equation*}
By the second-order Taylor expansion of $\frac{1}{2}\|Ax-b\|_{2}^{2}$
at $z^{k}$, we can derive that
\begin{equation*}
	\begin{split}
		\calF^k(\tilde{x}^{k})=\calE(\tilde{x}^k)-\frac{1}{2}(\tilde{x}^{k}-z^{k})^\top A^\top A(\tilde{x}^{k}-z^{k})+\frac{\beta}{2}\|\tilde{x}^{k}-z^{k}\|_{2}^{2}.
	\end{split}
\end{equation*}
Clearly, the following holds
\begin{equation}\label{eq_f1}
	\calF^k(\tilde{x}^{k})\leq \calE(\tilde{x}^k)+\frac{\beta}{2}\|\tilde{x}^{k}-z^{k}\|_{2}^{2}.
\end{equation}
Recalling $z^{k}$ in \eqref{eq_zkk}, one has
\begin{equation}\label{eq_f2}
	\begin{split}
		\|\tilde{x}^{k}-z^{k}\|_{2}^{2}=t_k^2\|\tilde{x}^{k}_{\tilde{\mathcal{S}}^{k}}-\tilde{x}^{k-1}_{\tilde{\mathcal{S}}^{k}}\|_{2}^{2}\leq t_k^2\|\tilde{x}^{k}-\tilde{x}^{k-1}\|_{2}^{2}.
	\end{split}
\end{equation}	
Combining \eqref{eq_f1} and \eqref{eq_f2}, we get
\begin{equation}\label{eq:mm-eq1}	
	\calF^k(\tilde{x}^{k})\leq\calE(\tilde{x}^k)+\frac{\beta}{2}t_k^2\|\tilde{x}^{k}-\tilde{x}^{k-1}\|_{2}^{2}.	
\end{equation}	

Now, we consider those $x$'s satisfying $\mathcal{S}(x) \subseteq \tilde{\mathcal{S}}^{k}$.
By the concavity of $\psi$ and the fact $\psi(\|x_{\slg_i}\|_p)=0,\;\forall \slg_i  \in \mathcal{G}\setminus\tilde{\mathcal{S}}^{k}$ when $\mathcal{S}(x) \subseteq \tilde{\mathcal{S}}^{k}$,
we have
\begin{equation}\label{eq_f3}
	\nonumber
	\begin{split}
		\calF^k(x)&\geq \frac{1}{2} \| Az^{k} - b \|_{2}^{2}
		+(x-z^{k})^\top A^\top(Az^{k}-b) + \frac{\beta}{2}\|x-z^{k}\|_{2}^{2} + \alpha \sum_{\slg_i  \in \tilde{\mathcal{S}}^{k}} \psi(\|x_{\slg_i}\|_p)\\
		&=\frac{1}{2} \| Az^{k} - b \|_{2}^{2}
		+(x-z^{k})^\top A^\top(Az^{k}-b) + \frac{\beta}{2}\|x-z^{k}\|_{2}^{2} + \alpha \sum_{\slg_i  \in \mathcal{G}} \psi(\|x_{\slg_i}\|_p).
	\end{split}
\end{equation}
Therefore, it holds that for $x$ satisfying $\mathcal{S}(x) \subseteq \tilde{\mathcal{S}}^{k}$,
\begin{equation}\label{eq:mm-eq2}
	\calF^k(x) \geq \calE(x)+ \frac{\beta-\|A^{\top}A\|_2}{2}\|x-z^{k}\|_{2}^{2}.
\end{equation}
%
Letting $x=x^{k+1}$ in \eqref{eq:mm-eq2}, we get
\begin{equation}\label{eq_f5}
	\calF^k(x^{k+1}) \geq \calE(x^{k+1})+ \frac{\beta-\|A^{\top}A\|_2}{2}\|x^{k+1}-z^{k}\|_{2}^{2}.
\end{equation}
By the definition of $z^{k}$, we have
\begin{equation}\label{eq_f6}
	\|x^{k+1}-z^{k}\|_{2}^{2}=\|x^{k+1}_{\tilde{\mathcal{S}}^{k}}-z^{k}_{\tilde{\mathcal{S}}^{k}}\|_{2}^{2}=\|(x^{k+1}_{\tilde{\mathcal{S}}^{k}}-\tilde{x}^{k}_{\tilde{\mathcal{S}}^{k}})-t_k(\tilde{x}^{k}_{\tilde{\mathcal{S}}^{k}}-\tilde{x}^{k-1}_{\tilde{\mathcal{S}}^{k}})\|_2^2.
\end{equation}
Applying the inequality $\|a-tb\|_2^2\geq (1-t)\|a\|_2^2+t(t-1)\| b\|_2^2~ (t\ge0)$ to \eqref{eq_f6} yields
\begin{equation}\label{eq_f7}
	\|(x^{k+1}_{\tilde{\mathcal{S}}^{k}}-\tilde{x}^{k}_{\tilde{\mathcal{S}}^{k}})-t_k(\tilde{x}^{k}_{\tilde{\mathcal{S}}^{k}}-\tilde{x}^{k-1}_{\tilde{\mathcal{S}}^{k}})\|_2^2\geq (1-t_k)\|x^{k+1}_{\tilde{\mathcal{S}}^{k}}-\tilde{x}^{k}_{\tilde{\mathcal{S}}^{k}}\|_2^2+t_k(t_k-1)\|\tilde{x}^{k}_{\tilde{\mathcal{S}}^{k}}-\tilde{x}^{k-1}_{\tilde{\mathcal{S}}^{k}}\|_2^2.
\end{equation}
Obviously, the following relationships hold
\begin{equation}\label{eq_relation}
	\|x^{k+1}_{\tilde{\mathcal{S}}^{k}}-\tilde{x}^{k}_{\tilde{\mathcal{S}}^{k}}\|_2^2=\|x^{k+1}-\tilde{x}^{k}\|_2^2,\; \|\tilde{x}^{k}_{\tilde{\mathcal{S}}^{k}}-\tilde{x}^{k-1}_{\tilde{\mathcal{S}}^{k}}\|_2^2\leq \|\tilde{x}^{k}-\tilde{x}^{k-1}\|_2^2.
\end{equation}
Combining \eqref{eq_f5}-\eqref{eq_relation} and using $t_k\in[0,\bar t), \beta>\|A^{\top}A\|_2$, we get
\begin{equation}\label{eq:mm-eqk}
	\begin{split}
		\calF^k(x^{k+1})&\geq \calE(x^{k+1})+\frac{\beta-\|A^{\top}A\|_2}{2}(1-t_{k})\|x^{k+1}-\tilde{x}^{k}\|_2^2\\
		&\quad+\frac{\beta-\|A^{\top}A\|_2}{2}t_k(t_k-1)\|\tilde{x}^{k}-\tilde{x}^{k-1}\|_2^2.
	\end{split}
\end{equation}

Since $\calJ_{\slg_i}^k(x_{\slg_i})$ is a $\beta$-strongly convex function and $x^{k+1}_{\slg_i}$ solves \eqref{eq:f} satisfying \eqref{eq_inexact_condition}, one has
\begin{equation*}
	\begin{split}
		\calJ_{\slg_i}^k(\tilde{x}^{k}_{\slg_i})\geq& \calJ_{\slg_i}^k(x^{k+1}_{\slg_i})+\langle\lambda^{k+1}_{\slg_i},\tilde{x}^{k}_{\slg_i}-x^{k+1}_{\slg_i}\rangle
		+\frac{\beta}{2}\|\tilde{x}^{k}_{\slg_i}-x^{k+1}_{\slg_i}\|_2^2\\
		\geq& \calJ_{\slg_i}^k(x^{k+1}_{\slg_i})-\|\lambda^{k+1}_{\slg_i}\|\|\tilde{x}^{k}_{\slg_i}-x^{k+1}_{\slg_i}\|
		+\frac{\beta}{2}\|\tilde{x}^{k}_{\slg_i}-x^{k+1}_{\slg_i}\|_2^2\\
		\geq& \calJ_{\slg_i}^k(x^{k+1}_{\slg_i})+\frac{\beta(1-\varepsilon)}{2}\|\tilde{x}^{k}_{\slg_i}-x^{k+1}_{\slg_i}\|_2^2.\\
	\end{split}
\end{equation*}
Furthermore, $\calJ_0^k(\tilde{x}^{k})=\calJ_0^k(x^{k+1})=\frac{1}{2} \| A z^k -b \|_{2}^{2}$.
Therefore, we have 
\begin{equation}\label{eq:suppequation2}
	\begin{split}
		\calF^k(\tilde{x}^{k})
		\geq& \calF^k(x^{k+1})+\frac{\beta(1-\varepsilon)}{2}\|\tilde{x}^{k}-x^{k+1}\|_2^2.\\
	\end{split}
\end{equation}
Combining~\eqref{eq:mm-eq1} \eqref{eq:mm-eqk} \eqref{eq:suppequation2}, we get
\begin{equation*}
	\begin{split}
		\calE(\tilde{x}^k)+\frac{\beta}{2}t_k^2\|\tilde{x}^{k}-\tilde{x}^{k-1}\|_{2}^{2}\geq & \calE(x^{k+1})+\frac{\beta-\|A^{\top}A\|_2}{2}(1-t_{k})\|x^{k+1}-\tilde{x}^{k}\|_2^2\\
		& +\frac{\beta-\|A^{\top}A\|_2}{2}t_k(t_k-1)\|\tilde{x}^{k}-\tilde{x}^{k-1}\|_2^2
		+\frac{\beta(1-\varepsilon)}{2}\|\tilde{x}^{k}-x^{k+1}\|_2^2\\
		\geq &\calE(x^{k+1})+\frac{\beta-\|A^{\top}A\|_2}{2}t_k(t_k-1)\|\tilde{x}^{k}-\tilde{x}^{k-1}\|_2^2\\
		&+\frac{\beta(1-\varepsilon)}{2}\|\tilde{x}^{k}-x^{k+1}\|_2^2.
	\end{split}
\end{equation*}
%
	By the definition of $\mathcal{H}$ in \eqref{eq_ht}, 
	the above inequality can be written as
	\begin{equation*}
		\mathcal{H}(\tilde{x}^{k},\tilde{x}^{k-1})-\mathcal{H}(x^{k+1},\tilde{x}^{k})\geq \left(-\frac{\|A^{\top}A\|_2}{2}t_{k}^2-\frac{\beta-\|A^{\top}A\|_2}{2}t_{k}
		+\frac{\beta(1-\varepsilon)}{2}\right)\|\tilde{x}^{k}-\tilde{x}^{k-1}\|_2^2.
	\end{equation*}
	
	Since $\tilde{x}^{k}=x^{k}$ and $\tilde{x}^{k-1}=x^{k-1}$ hold when $k\geq K+1$, we have, for $\forall k\geq K+1$,
	\begin{equation}\label{eq:sufficient-decrease-condition_copy2}
		\mathcal{H}(x^{k},x^{k-1})-\mathcal{H}(x^{k+1},x^{k})\geq \left(-\frac{\|A^{\top}A\|_2}{2}t_{k}^2-\frac{\beta-\|A^{\top}A\|_2}{2}t_{k}
		+\frac{\beta(1-\varepsilon)}{2}\right)
		\|x^{k}-x^{k-1}\|_2^2.
	\end{equation}
	Due to our assumption that $\set{t_{k}}\subset[0,{\bar t})$ and  $\hat{t}=\sup_{k} t_k<{\bar t}$, the following holds
	\begin{equation}\label{eq:lowbound}
		\begin{split}
		-\frac{\|A^{\top}A\|_2}{2}t_{k}^2-\frac{\beta-\|A^{\top}A\|_2}{2}t_{k}+\frac{\beta(1-\varepsilon)}{2}&\geq \left(-\frac{\|A^{\top}A\|_2}{2}\hat{t}^2-\frac{\beta-\|A^{\top}A\|_2}{2}\hat{t}
		+\frac{\beta(1-\varepsilon)}{2}\right)\\
		&>0.
		\end{split}
	\end{equation}
	Therefore, for any $k\geq K+1$, the sequence $\set{\mathcal{H}(x^{k},x^{k-1})}$ is monotonically nonincreasing. 
	
	$(ii)$
	Summing up both sides of \eqref{eq:sufficient-decrease-condition_copy2} from $k=K+1$ to $N$ and using \eqref{eq:lowbound}, we get
	\begin{equation*}
		\begin{split}
		&\left(-\frac{\|A^{\top}A\|_2}{2}\hat{t}^2-\frac{\beta-\|A^{\top}A\|_2}{2}\hat{t}+
		\frac{\beta(1-\varepsilon)}{2}\right)\sum_{k=K+1}^{N}\| x^{k} - x^{k-1} \|_{2}^{2}\\
		&\leq\mathcal{H}(x^{K+1},x^{K})-\mathcal{H}(x^{N+1},x^{N}).
		\end{split}
	\end{equation*}
	Let $N \to \infty$ yields
	\begin{equation*}
		\left(-\frac{\|A^{\top}A\|_2}{2}\hat{t}^2-\frac{\beta-\|A^{\top}A\|_2}{2}\hat{t}
		+\frac{\beta(1-\varepsilon)}{2}\right)\sum_{k=K+1}^{\infty}\| x^{k} - x^{k-1} \|_{2}^{2}\leq\mathcal{H}(x^{K+1},x^{K}).
	\end{equation*}
	It follows that $\sum_{k=0}^{\infty}\| x^{k+1} - x^{k} \|_{2}^{2}<+\infty$. Thus
	$\lim_{k \to \infty} \| x^{k+1} - x^{k} \|_{2} = 0$. For any $k\geq K+1$, the sequence $\set{\mathcal{H}(x^{k},x^{k-1})}$ is monotonically nonincreasing, so we have
	$$
	\calE(x^k)\leq \mathcal{H}(x^{k},x^{k-1})\leq \mathcal{H}(x^{K+1},x^{K}),\; \forall k\geq K+1.
	$$
	Therefore, the sequence $\set{x^k}_{k\geq K+1}$ is bounded due to  the coercivity of $\calE(x)$. Furthermore, the whole sequence $\set{x^k}$ is also bounded, because $K$ is a finite integer.
	
	
	$(iii)$  The nonnegative sequence $\set{\mathcal{H}(x^{k},x^{k-1})}_{k\geq K+1}$ is monotonically nonincreasing, so  $\set{\mathcal{H}(x^{k},x^{k-1})}$ is convergent. Then, the sequence $\set{\calE(x^{k})}$ is also convergent due to  $\lim_{k \to \infty} \| x^{k+1}$ $ - x^{k} \|_{2} = 0$. Thus $\zeta:=\lim_{k \to \infty}\calE(x^k)$ exists. Since $\calE$ is a continuous funciton, one has $\calE\equiv \zeta$ on $\Omega$.
	
	$(iv)$ Since $\lim_{k \to \infty}\calE(x^k)=\zeta$ and $\lim_{k \to \infty} \| x^{k+1} - x^{k} \|_{2} = 0$, thus $\lim_{k \to \infty}\mathcal{H}(x^k,x^{k-1})=\zeta$. Since $\mathcal{H}$ is a continuous function, then $\calH\equiv \zeta$ on $\Upsilon$.
	This completes the proof.
\end{proof}

Before proceeding, we calculate the subdifferential of $\calJ_{\slg_i}^k(x_{\slg_i})$, $\slg_{i}\in \tilde{\mathcal{S}}^{k}$ to consider \eqref{eq_inexact_condition}. In particular, we have
\begin{equation}\label{eq_optimalcondition}
	\lambda^{k+1}_{\slg_{i}}\in (A^\top (Az^{k}-b))_{\slg_{i}}+\beta (x^{k+1}_{\slg_{i}}-z^{k}_{\slg_{i}})+\alpha \psi^{\prime} (\|\tilde{x}_{\slg_{i}}^{k}\|_p) \partial \|\cdot\|_p(x^{k+1}_{\slg_{i}}),\; \forall \slg_{i}\in \tilde{\mathcal{S}}^{k}.
\end{equation}
Since for any $k\geq K$, $\tilde{\mathcal{S}}^{k}=\mathcal{S}^{k}$ and $\tilde{x}^{k} = x^{k}$ hold, by \eqref{eq_subdiff},  \eqref{eq_optimalcondition} indicates
\begin{itemize}
	\item
	For $p>1,\forall j\in \slg_i$ with $\slg_i\in \mathcal{S}^{k}, k \geq K$, we have
	\begin{equation}\label{eq_optcon2}
		\begin{split}
		\lambda^{k+1}_{j} =&(A^\top(Az^k-b))_{j}+ \beta(x^{k+1}-z^k)_{j}\\
	&+ \alpha \psi^{\prime} (\|x_{\slg_i}^{k}\|_p)\|x_{\slg_{i}}^{k+1}\|_p^{1-p}|x_{j}^{k+1}|^{p-1}\sgn(x_{j}^{k+1});
		\end{split}
	\end{equation}
	\item
	For $p=1,\forall j\in s(x^{k+1}_{\slg_i})$ with $ \slg_i  \in \mathcal{S}^{k}, k \geq K$, we have
	\begin{equation}\label{eq_optcon3}
		\begin{split}
			\lambda^{k+1}_{j}=(A^\top(Az^k-b))_{j} + \beta(x^{k+1}-z^k)_{j}
			+ \alpha \psi^{\prime} (\|x_{\slg_{i}}^{k}\|_1)\sgn(x_{j}^{k+1});
		\end{split}
	\end{equation}
	\item
	For $p=1,\forall j\in\slg_i\backslash s(x^{k+1}_{\slg_i})$ with $ \slg_i  \in \mathcal{S}^{k}, k \geq K$, we have
	\begin{equation}\label{eq_optcon4}
		\begin{split}
			\lambda^{k+1}_{j}\in(A^\top(Az^{k}-b))_{j} + \beta(x^{k+1}-z^k)_{j}
			+ \alpha [-\psi^{\prime} (\|x_{\slg_{i}}^{k}\|_1),\psi^{\prime} (\|x_{\slg_{i}}^{k}\|_1)].
		\end{split}
	\end{equation}
\end{itemize}
The next lemma provides a lower bound and upper bound theory for the iterative sequence $\set{x^k}$.
\begin{lemma}\label{lem-bound-theory}
	For any $\beta>\|A^{\top}A\|_2$ and the extrapolation parameters $\set{t_{k}}\subset[0,{\bar t})$ satisfying $\sup_{k} t_k<{\bar t}$,
	let $\set{x^{k}}$ be a sequence generated by FITS$^3$.
	Then for any $k\geq K+1$,
	\begin{enumerate}[(i)]
		\item  $\|x_{\slg_i}^{k}\|_p\le \overline {c}:=\psi^{-1}\left( \frac{\mathcal{H}(x^{K+1},x^{K})}{\alpha}\right)$, for any $\slg_i \in \mathcal{G}$.
		\item $\|x_{\slg_i}^{k}\|_p\ge\underline {c}:=\max\left\{(\psi^{\prime})^{-1}\left(\frac{ (2\|A^\top A\|_2+\beta(4+\varepsilon)+2\|A \|_2\sqrt{\beta(1-\varepsilon)})\sqrt{n\mathcal{H}(x^{K+1},x^{K})}}
		{\alpha\sqrt{2\beta(1-\varepsilon)}}\right), \tau\right\}$, for any $\slg_{i} \in \mathcal{S}^{k}$.
	\end{enumerate}
\end{lemma}

\begin{proof}
	Since in Lemma~\ref{lem-bound-sufficient-decrease}, the sequence $\set{\mathcal{H}(x^{k},x^{k-1})}$ is nonincreasing for any $k\geq K+1$, one has
	\begin{equation}\label{eq:cC}
		\begin{split}
			\mathcal{H}(x^{k},x^{k-1})\leq \mathcal{H}(x^{K+1},x^{K}),\; \forall k\geq K+1,
		\end{split}
	\end{equation}
	which indicates $\alpha\psi(\|x_{\slg_i}^{k}\|_p)\leq \mathcal{H}(x^{K+1},x^{K})$. Then, one has
	\begin{equation}\label{eq:C}
		\nonumber
		\begin{split}
			\|x_{\slg_i}^{k}\|_p\leq \psi^{-1}\left( \frac{\mathcal{H}(x^{K+1},x^{K})}{\alpha}\right):=\overline {c},\; \forall k\geq K+1.
		\end{split}
	\end{equation}
	
	We next prove that any $\|x_{\slg_i}^{k} \|_p\neq 0$ with $\slg_i \in \mathcal{S}^{k}, k\geq K$ has a uniform lower bound. Obviously, $\|x_{\slg_i}^{k} \|_p$ with $\slg_i \in \mathcal{S}^{k}, k\geq K$ has a lower bound $\tau$. In the following, we show that $\|x_{\slg_i}^{k} \|_p$ with $\slg_i \in \mathcal{S}^{k}, k\geq K$  has also a positive lower bound independent of $\tau$.  We discuss two cases with $p>1$ and $p=1$ separately.
	
	Firstly, we consider that $p>1$.
	It follows from \eqref{eq_optcon2} that for any $j\in \slg_{i}$ with $\slg_{i} \in \mathcal{S}^{k}, k\geq K$,
	\begin{equation*}
		\alpha \psi^{\prime} (\|x_{\slg_{i}}^{k}\|_p)\|x_{\slg_{i}}^{k+1}\|_p^{1-p}|x_{j}^{k+1}|^{p-1}=\left|\left(A^\top(Az^{k}-b) + \beta(x^{k+1}-z^{k})-\lambda^{k+1}\right)_{j}\right|.
	\end{equation*}
	Adding up the above equations for all $j\in \slg_{i}$ with $\slg_{i} \in \mathcal{S}^{k}, k\geq K$, we get
	\begin{equation*}
		\alpha \psi^{\prime} (\|x_{\slg_{i}}^{k}\|_p)\|x_{\slg_{i}}^{k+1}\|_p^{1-p}\|x_{\slg_{i}}^{k+1}\|_{p-1}^{p-1}
		=\left\|\left(A^\top(Az^{k}-b) + \beta(x^{k+1}-z^{k})-\lambda^{k+1}\right)_{\slg_{i}}\right\|_1.
	\end{equation*}
	Due to	Lemma \ref{lemma_inequa}, we see $\|x_{\slg_{i}}^{k+1}\|_p^{1-p}\|x_{\slg_{i}}^{k+1}\|_{p-1}^{p-1}\geq 1$.	
	Therefore, by Lemma \ref{lemma_ineq2},	we have for $\slg_{i} \in \mathcal{S}^{k}, k\geq K+1$ that
	\begin{equation}\label{eq:L4-eqqq}
		\begin{split}
			\alpha \psi^{\prime} (\|x_{\slg_{i}}^{k}\|_p)\leq&\left\|\left(A^\top(Az^{k}-b) + \beta(x^{k+1}-z^{k})-\lambda^{k+1}\right)_{\slg_{i}}\right\|_1\\
			\leq&\|A^\top(Az^{k}-b)\|_1 + \beta\|x^{k+1}-z^{k})\|_1+\|\lambda^{k+1}_{{\mathcal{S}}^{k}}\|_1\\
			[~\text{by~\eqref{eq_inexact_condition2}}~]~\le&\sqrt{n} \|A^\top(Az^{k}-Ax^{k}+Ax^{k}-b)\|_{2}+\sqrt{n}\beta\|(x^{k+1}_{\mathcal{S}^{k}}-
			x^{k}_{\mathcal{S}^{k}})-t_k(x^{k}_{\mathcal{S}^{k}}-x^{k-1}_{\mathcal{S}^{k}})\|_2\\
			&+\frac{\beta\varepsilon}2\sqrt{n}\|x^{k+1}-\tilde x^k\|_2\\
			\leq& \sqrt{n}(\|A^\top A\|_2\|x^{k}-x^{k-1}\|_2+\|A\|_2\|Ax^{k}-b\|_2)+\sqrt{n}\beta (\|x^{k+1}-x^{k}\|_2\\
			&+\|x^{k}-x^{k-1}\|_2)+\frac{\beta\varepsilon}2\sqrt{n}\|x^{k+1}- x^k\|_2.\\
		\end{split}
	\end{equation}
	By \eqref{eq:cC}, one has $\|x^{k+1}-x^{k}\|_2\leq \sqrt{2\mathcal{H}(x^{K+1},x^{K})/\beta(1-\varepsilon)}$, $\|x^{k}-x^{k-1}\|_2\leq \sqrt{2\mathcal{H}(x^{K+1},x^{K})/\beta(1-\varepsilon)}$ and $\|Ax^{k}-b\|_2\leq \sqrt{2\mathcal{H}(x^{K+1},x^{K})}$ for $k\ge K+1$.
	Therefore, we get for $k\ge K+1$,
	\begin{equation*}
		\psi^{\prime} (\|x_{\slg_{i}}^{k}\|_p)\leq\frac{ \sqrt{n}(2\|A^\top A\|_2+\beta(4+\varepsilon)+2\|A \|_2\sqrt{\beta(1-\varepsilon)})\sqrt{\mathcal{H}(x^{K+1},x^{K})}}{\alpha\sqrt{2\beta(1-\varepsilon)}}.
	\end{equation*}
	Since $\psi$ is a concave function,  one has for $k\ge K+1$,
	\begin{equation*}
		\|x_{\slg_{i}}^{k}\|_p\geq (\psi^{\prime})^{-1}\left(\frac{ \sqrt{n}(2\|A^\top A\|_2+\beta(4+\varepsilon)+2\|A \|_2\sqrt{\beta(1-\varepsilon)})\sqrt{\mathcal{H}(x^{K+1},x^{K})}}{\alpha\sqrt{2\beta(1-\varepsilon)}}\right).
	\end{equation*}
	
	%
	%
	%
	%
	Next, we consider the case that $p=1$.  \eqref{eq_optcon3}  implies  for any $j\in s(x^{k+1}_{\slg_{i}})$ with $\slg_{i} \in \mathcal{S}^{k}, k\geq K+1$,
	\begin{equation*}
		\begin{split}
			\alpha \psi^{\prime} (\|x_{\slg_{i}}^{k}\|_1)
			\leq&\|A^\top(Az^{k}-b)\|_1 + \beta\|x^{k+1}-z^{k})\|_1+\|\lambda^{k+1}_{{\mathcal{S}}^{k}}\|_1,\\
		\end{split}
	\end{equation*}
	which is exactly a similar inequality as the second inequality in  \eqref{eq:L4-eqqq}. We can therefore derive  for $k\ge K+1$,
	\begin{equation*}
		\|x_{\slg_{i}}^{k}\|_1\geq(\psi^{\prime})^{-1}\left(\frac{ \sqrt{n}(2\|A^\top A\|_2+\beta(4+\varepsilon)+2\|A \|_2\sqrt{\beta(1-\varepsilon)})\sqrt{\mathcal{H}(x^{K+1},x^{K})}}{\alpha\sqrt{2\beta(1-\varepsilon)}}\right).
	\end{equation*}
	This completes the proof.
\end{proof}

Lemma~\ref{lem-bound-theory} and Remark~\ref{property_psi}(i) indicate that  there exists $L_{\underline {c}} > 0$ such that for any $ \slg_i \in \mathcal{S}^{k},\;k \geq K+1$,

\begin{equation}\label{eq:grad-Lip-cond-K}
	\abs{ \psi^{\prime}(\|x_{\slg_i}^{k+1}\|_p) - \psi^{\prime}(\|x_{\slg_i}^{k}\|_p) }
	\leq L_{\underline {c}}\abs{\|x_{\slg_i}^{k+1}\|_p - \|x_{\slg_i}^{k}\|_p}
	\leq L_{\underline {c}}\|x_{\slg_i}^{k+1}-x_{\slg_i}^{k}\|_p.
\end{equation}
\begin{lemma}\label{lem-relative-error-condition}
	For the extrapolation parameters $\set{t_{k}}\subset[0,{\bar t})$, 
	let $\set{x^{k}}$ be a sequence generated by FITS$^3$. Then for each $k\geq K+1$, there exist $\nu^{k+1}  \in \partial \calE(x^{k+1})$ such that
	\begin{equation}\label{eq:relative-error-condition}
		\| \nu^{k+1} \|_{2}
		\leq  \Gamma_1\|x^{k+1} - x^{k} \|_{2}+\Gamma_2\|x^{k} - x^{k-1} \|_{2},
	\end{equation}
	and $\gamma^{k+1}\in \partial\mathcal{H}(x^{k+1},x^{k})$ such that
	\begin{equation}
		\|\gamma^{k+1}\|_2\leq \Gamma_3 \|x^{k+1} - x^{k} \|_{2}+\Gamma_2\|x^{k} - x^{k-1} \|_{2},
	\end{equation}	
	with
	$\Gamma_1=\max\left\{(\|A^{\top}A\|_2+\beta+\alpha (\widetilde{C}_{p-1})^{p-1}L_{\underline {c}}+\frac{\beta\varepsilon}2)\sqrt{n}, (\|A^{\top}A\|_2+\beta+\frac{\beta\varepsilon}2+n\alpha L_{\underline {c}}) \sqrt{n}\right\}$,
$\Gamma_2=(\|A^{\top}A\|_2+\beta)\sqrt{n},\ \Gamma_3=\Gamma_1+\sqrt{2}\beta(1-\varepsilon), 
	$
	and a constant $\widetilde{C}_{p-1}$ defined in \eqref{eq_ineq3}.
\end{lemma}

\begin{proof}
	We will discuss  $p>1$ and $p=1$ separately.
	
	(1) The case of $p>1$.	
	Let $k\geq K+1$ and define $\nu^{k+1}$ by
	\begin{equation}\label{eq_subdiffff}
		\small
		\left\{
		\begin{aligned}
			\nu_{j}^{k+1} &=(A^\top(Ax^{k+1}-b))_{j}
			+ \alpha \psi^{\prime} (\|x_{\slg_i}^{k+1}\|_p)\|x_{\slg_{i}}^{k+1}\|_p^{1-p}|x_{j}^{k+1}|^{p-1}\sgn(x_{j}^{k+1}), \;\forall j\in \slg_i\, , \,\slg_i  \in \mathcal{S}^{k}; \\
			\nu_{j}^{k+1} &= 0, \quad\forall j\in \slg_i\, , \,\slg_i \in \mathcal{G} \setminus \mathcal{S}^{k}.
		\end{aligned}\right.
	\end{equation}
	According to the subdifferential in \eqref{eq_subdiff}, one has $\nu^{k+1}\in\partial\calE(x^{k+1})$. Recall \eqref{eq_optcon2} that for any $j\in \slg_i$ with $ \slg_i\in \mathcal{S}^{k}, k\geq K$
	\begin{equation}\label{eq:L4-subdiff1}
		\begin{split}
			\lambda^{k+1}_j =(A^\top(Az^{k}-b))_{j}+ \beta(x^{k+1}-z^{k})_{j}
			+ \alpha \psi^{\prime} (\|x_{\slg_i}^{k}\|_p)\|x_{\slg_{i}}^{k+1}\|_p^{1-p}|x_{j}^{k+1}|^{p-1}\sgn(x_{j}^{k+1}).
		\end{split}
	\end{equation}
	Combining \eqref{eq_subdiffff} and \eqref{eq:L4-subdiff1}, we obtain that for any $j\in\slg_i$ with $ \slg_i  \in \mathcal{S}^{k}, k\geq K+1$
	\begin{equation}\label{eq:L4-subdiff2}
		\small
		\begin{split}
			\nu_{j}^{k+1} =&(A^\top A(x^{k+1}-z^{k}))_{j}
			-\beta(x^{k+1}-z^{k})_{j}+\lambda^{k+1}_j\\
			&+\alpha \|x_{\slg_{i}}^{k+1}\|_p^{1-p}|x_{j}^{k+1}|^{p-1}\sgn(x_{j}^{k+1})\left(\psi^{\prime} (\|x_{\slg_i}^{k+1}\|_p)-\psi^{\prime}(\|x_{\slg_i}^{k}\|_p)\right).
		\end{split}
	\end{equation}
	From \eqref{eq:L4-subdiff2} and using \eqref{eq:grad-Lip-cond-K}, we get
	\begin{equation}\label{eq_nu}
		\small
		\begin{split}
			|\nu_{j}^{k+1}|\leq|(A^\top A(x^{k+1}-z^{k}))_{j}|+\beta|(x^{k+1}-z^{k})_{j}|+|\lambda^{k+1}_j|
			+\alpha L_{\underline {c}} \|x_{\slg_{i}}^{k+1}\|_p^{1-p}|x_{j}^{k+1}|^{p-1}\|x_{\slg_i}^{k+1}-x_{\slg_i}^{k}\|_p.
		\end{split}
	\end{equation}
	Adding up all $j\in \slg_i$ for $\slg_i  \in \mathcal{S}^{k}, k\geq K+1$ in \eqref{eq_nu} yields
	\begin{equation}\label{eq_subdiffl1}
		\begin{split}
		\|\nu_{\slg_i}^{k+1}\|_1\leq&\|(A^\top A(x^{k+1}-z^{k}))_{\slg_i}\|_1+\beta\|(x^{k+1}-z^{k})_{\slg_i}\|_1+\|\lambda^{k+1}_{\slg_i}\|_1\\
		&+\alpha L_{\underline{c}}\|x_{\slg_{i}}^{k+1}\|_p^{1-p}\|x_{\slg_i}^{k+1}\|_{p-1}^{p-1}
		\|x_{\slg_i}^{k+1}-x_{\slg_i}^{k}\|_p.
		\end{split}
	\end{equation}
	By Lemma \ref{lemma_inequa} and Lemma \ref{lemma_ineq2}, one has
	\begin{equation}\label{eq_Cp-1}
		\begin{split}
		\alpha L_{\underline {c}}\|x_{\slg_{i}}^{k+1}\|_p^{1-p}\|x_{\slg_i}^{k+1}\|_{p-1}^{p-1}\|x_{\slg_i}^{k+1}-x_{\slg_i}^{k}\|_p
		&\leq \alpha  (\widetilde{C}_{p-1})^{p-1}L_{\underline {c}}\|x_{\slg_i}^{k+1}-x_{\slg_i}^{k}\|_p\\
		&\leq\alpha  (\widetilde{C}_{p-1})^{p-1}L_{\underline {c}}\|x_{\slg_i}^{k+1}-x_{\slg_i}^{k}\|_1,
	\end{split}
\end{equation}
	where $\widetilde{C}_{p-1}$ is a constant.
	Combining \eqref{eq_subdiffl1} and \eqref{eq_Cp-1} and using Lemma \ref{lemma_ineq2}, we get
	\begin{equation}
		\nonumber
		\begin{split}
			&\|\nu^{k+1}\|_1\\ 
			& \leq  \sqrt{n}\left(\|A^\top A(x^{k+1}-z^{k})\|_2+ \beta\|x^{k+1}-z^{k}\|_2+\|\lambda^{k+1}_{{\mathcal{S}}^{k}}\|_2+\alpha   (\widetilde{C}_{p-1})^{p-1}L_{\underline {c}}\|x^{k+1}-x^{k}\|_2\right).
		\end{split}
	\end{equation}	
	Recalling  \eqref{eq_inexact_condition2} and the definition of $z^k$ in \eqref{eq_zkk}, one has for $k\ge K+1$ that
	\begin{equation}\label{eq_subdiffl3}
		\nonumber
		\begin{split}
			\|\nu^{k+1}\|_1\leq&
			\left(\|A^{\top}A\|_2+\beta+\alpha (\widetilde{C}_{p-1})^{p-1} L_{\underline {c}}+\frac{\beta\varepsilon}2\right) \sqrt{n}\|x^{k+1}-x^{k}\|_2\\
		&	+(\|A^{\top}A\|_2+\beta) \sqrt{n}\|x^{k}-x^{k-1}\|_{2}.
		\end{split}
	\end{equation}
	Using the inequality in Lemma \ref{lemma_inequa}, we get
	\begin{equation}\label{eq_impt}
		\begin{split}
		\|\nu^{k+1}\|_2\leq& \left(\|A^{\top}A\|_2+\beta+\alpha (\widetilde{C}_{p-1})^{p-1}L_{\underline {c}}+\frac{\beta\varepsilon}2\right) \sqrt{n}\|x^{k+1}-x^{k}\|_2\\
	&	+(\|A^{\top}A\|_2+\beta) \sqrt{n}\|x^{k}-x^{k-1}\|_{2}.
		\end{split}
	\end{equation}

	Again for $k\ge K+1$, we	define $\gamma^{k+1}:=(\gamma^{k+1}_x,\gamma^{k+1}_u)$ by
	\begin{equation}\label{eq_rxry}
		\left\{
		\begin{aligned}
			\gamma^{k+1}_x & =	\nu^{k+1}+\beta(1-\varepsilon)(x^{k+1}-x^{k});\\
			\gamma^{k+1}_u & =-\beta(1-\varepsilon)(x^{k+1}-x^{k}).
		\end{aligned}
		\right.
	\end{equation}
	By the definition of $\mathcal{H}(x,u)$ in \eqref{eq_ht}, we have $\gamma^{k+1}\in\partial\mathcal{H}(x^{k+1},x^{k})$. Meanwhile, it holds that
	\begin{equation*}
		\|\gamma^{k+1}\|_2 \leq \|\nu^{k+1}\|_2+\sqrt{2}\beta(1-\varepsilon)\|x^{k+1}-x^{k}\|_2.
	\end{equation*}
	Therefore, using \eqref{eq_impt}, one has
	\begin{equation}\label{eq_gamma}
		\begin{aligned}
			\|\gamma^{k+1}\|_2\leq &\left(\left(\|A^{\top}A\|_2+\beta+\alpha (\widetilde{C}_{p-1})^{p-1}L_{\underline {c}}+\frac{\beta\epsilon}2\right)\sqrt{n}+\sqrt{2}\beta(1-\varepsilon)\right)\|x^{k+1}-x^{k}\|_2\\
			&+(\|A^{\top}A\|_2+\beta)\sqrt{n}\|x^{k}-x^{k-1}\|_2.
		\end{aligned}
	\end{equation}

	(2) The case of $p=1$. Let $k\ge K+1$ and define $\nu^{k+1}$ as
	\begin{equation}\label{eq_sudiff_p1}
		\nu_{j}^{k+1}=\left\{
		\begin{aligned}
			&(A^\top(Ax^{k+1}-b))_{j}
			+ \alpha \psi^{\prime} (\|x_{\slg_i}^{k+1}\|_1)\sgn(x_{j}^{k+1}), \quad\forall j\in s(x^{k+1}_{\slg_i}),\;  \slg_i  \in \mathcal{S}^{k}; \\
			& (A^\top(Ax^{k+1}-b))_{j}+\alpha \xi_{j},\;\;\,\;\qquad\quad\quad\qquad\qquad\quad\forall j\notin s(x^{k+1}_{\slg_i}),\;  \slg_i  \in \mathcal{S}^{k}; \\
			& 0, \qquad\qquad\qquad\qquad\qquad\qquad\qquad\qquad\qquad\quad\;\;\;\;\forall j\in \slg_i, \;\slg_i \in \mathcal{G} \setminus \mathcal{S}^{k}.
		\end{aligned}\right.
	\end{equation}
	where $\xi_{j}\in[-\psi^{\prime} (\|x_{\slg_i}^{k+1}\|_1),\psi^{\prime} (\|x_{\slg_i}^{k+1}\|_1)]$ will be selected later. By the subdifferential in \eqref{eq_subdiff}, one has $\nu^{k+1}\in\partial\calE(x^{k+1})$.  Recall  \eqref{eq_optcon3} and \eqref{eq_optcon4} that for any $j\in s(x^{k+1}_{\slg_i})$ with $ \slg_i  \in \mathcal{S}^{k}, k\geq K$,
	\begin{equation}\label{eq_optp1}
		\begin{split}
			\lambda^{k+1}_j=(A^\top(Az^{k}-b))_{j} + \beta(x^{k+1}-z^{k})_{j}
			+ \alpha \psi^{\prime} (\|x_{\slg_{i}}^{k}\|_1)\sgn(x_{j}^{k+1}),
		\end{split}
	\end{equation}
	and  that for any $j\not\in s(x^{k+1}_{\slg_i})$ with $ \slg_i  \in \mathcal{S}^{k}, k\geq K$,
	\begin{equation}
		\label{eq_optp1ns}
		\begin{split}
			\lambda^{k+1}_j\in(A^\top(Az^{k}-b))_{j} + \beta(x^{k+1}-z^{k})_{j}
			+ \alpha [-\psi^{\prime} (\|x_{\slg_{i}}^{k}\|_1),\psi^{\prime} (\|x_{\slg_{i}}^{k}\|_1)].
		\end{split}
	\end{equation}
	
	On one hand, combining \eqref{eq_sudiff_p1} and \eqref{eq_optp1}, we get for any $j\in s(x^{k+1}_{\slg_i})$ with $ \slg_i  \in \mathcal{S}^{k}, k\geq K+1$,
	\begin{equation*}\label{eq:L4-subdiff21}
		\begin{split}
			\nu_{j}^{k+1} =&(A^\top A(x^{k+1}-z^{k}))_{j}
			-\beta(x^{k+1}-z^{k})_{j}+\lambda^{k+1}_j\\
			&+\alpha \sgn(x_{j}^{k+1})\left(\psi^{\prime} (\|x_{\slg_i}^{k+1}\|_1)-\psi^{\prime}(\|x_{\slg_i}^{k}\|_1)\right).
		\end{split}
	\end{equation*}
	Using \eqref{eq:grad-Lip-cond-K}, it then follows that for any $j\in s(x^{k+1}_{\slg_i})$ with $ \slg_i  \in \mathcal{S}^{k}, k\geq K+1$,
	\begin{equation}\label{eq_vp0}
		\begin{split}
			|\nu_{j}^{k+1}|&\leq  |(A^\top A(x^{k+1}-z^{k}))_{j}|
			+\beta|(x^{k+1}-z^{k})_{j}|+|\lambda^{k+1}_j|+\alpha L_{\underline {c}} \|x_{\slg_i}^{k+1}-x_{\slg_i}^{k}\|_1.
		\end{split}
	\end{equation}

	On the other hand, by \eqref{eq_optp1ns}, there exists an $\eta_{j}\in [-\psi^{\prime} (\|x_{\slg_{i}}^{k}\|_1),\psi^{\prime} (\|x_{\slg_{i}}^{k}\|_1)]$ such that
	\begin{equation}\label{eq_q1opt}
		\begin{split}
			\nonumber
			\lambda^{k+1}_j=(A^\top(Az^{k}-b))_{j} + \beta(x^{k+1}-z^{k})_{j}
			+ \alpha \eta_{j},\;j\not\in s(x^{k+1}_{\slg_i}).
		\end{split}
	\end{equation}
	Then, for any $j\notin s(x^{k+1}_{\slg_i})$ with $ \slg_i  \in \mathcal{S}^{k}, k\geq K+1$, $\nu_{j}^{k+1}$ in \eqref{eq_sudiff_p1} could be written as
	\begin{equation}\label{eq_nujk1}
		\begin{split}
			\nu_{j}^{k+1} =(A^\top A(x^{k+1}-z^{k}))_{j}
			-\beta(x^{k+1}-z^{k})_{j}+\lambda^{k+1}_j+\alpha (\xi_{j}-\eta_{j}).
		\end{split}
	\end{equation}
	In order to estimate the value of $|\nu_{j}^{k+1}|$ for any $j\notin s(x^{k+1}_{\slg_i})$ with $ \slg_i  \in \mathcal{S}^{k}, k\geq K$, we use the method in \cite{xue2019efficient} to choose
	$\xi_j$ in \eqref{eq_sudiff_p1}.
	If $\psi^{\prime} (\|x_{\slg_{i}}^{k}\|_1)\leq \psi^{\prime} (\|x_{\slg_{i}}^{k+1}\|_1)$, we set $\xi_{j}=\eta_{j}$. If $\psi^{\prime} (\|x_{\slg_{i}}^{k}\|_1)> \psi^{\prime} (\|x_{\slg_{i}}^{k+1}\|_1)$, we set
	\begin{equation*}
		\xi_{j}=\left\{
		\begin{aligned}
			\eta_{j},& \quad \eta_{j}\in \left[-\psi^{\prime} (\|x_{\slg_{i}}^{k+1}\|_1),\psi^{\prime} (\|x_{\slg_{i}}^{k+1}\|_1)\right];\\
			\psi^{\prime} (\|x_{\slg_{i}}^{k+1}\|_1),&\quad \eta_{j}\in \left(\psi^{\prime} (\|x_{\slg_{i}}^{k+1}\|_1),\psi^{\prime} (\|x_{\slg_{i}}^{k}\|_1)\right];\\
			-\psi^{\prime} (\|x_{\slg_{i}}^{k+1}\|_1), &\quad \eta_{j}\in \left[-\psi^{\prime} (\|x_{\slg_{i}}^{k}\|_1),-\psi^{\prime} (\|x_{\slg_{i}}^{k+1}\|_1)\right).
		\end{aligned}\right.
	\end{equation*}
	By the choice of $\xi_{j}$ and \eqref{eq_nujk1}, the following holds for  $j\notin s(x^{k+1}_{\slg_i})$ with $ \slg_i  \in \mathcal{S}^{k}, k\geq K+1$,
	\begin{equation}\label{eq_vp1}
		\begin{split}
			|\nu_{j}^{k+1}|&\leq |(A^\top A(x^{k+1}-z^{k}))_{j}|
			+\beta|(x^{k+1}-z^{k})_{j}|+|\lambda^{k+1}_j|+\alpha \left|\psi^{\prime} (\|x_{\slg_i}^{k+1}\|_1)-\psi^{\prime}(\|x_{\slg_i}^{k}\|_1)\right|\\
			&\leq |(A^\top A(x^{k+1}-z^{k}))_{j}|
			+\beta|(x^{k+1}-z^{k})_{j}|+|\lambda^{k+1}_j|+\alpha L_{\underline {c}} \|x_{\slg_i}^{k+1}-x_{\slg_i}^{k}\|_1.
		\end{split}
	\end{equation}
	
	Adding up all $j\in \slg_i$, for $ \slg_i  \in \mathcal{S}^{k}, k\geq K+1$ in \eqref{eq_vp0} and \eqref{eq_vp1}, we obtain
	\begin{equation}\label{eq_s}
		\nonumber
		\begin{split}
			\|\nu_{\slg_i}^{k+1}\|_1\leq\|(A^\top A(x^{k+1}-z^{k}))_{\slg_i}\|_1
			+\beta\|(x^{k+1}-z^{k})_{\slg_i}\|_1+\|\lambda^{k+1}_{\slg_i}\|_1+n\alpha L_{\underline {c}} \|x_{\slg_i}^{k+1}-x_{\slg_i}^{k}\|_1.
		\end{split}
	\end{equation}
	Recalling the construction of $\nu^{k+1}$ in \eqref{eq_sudiff_p1}, one has $\nu^{k+1}_j=0$, for any $j\in \slg_i$ with $ \slg_i \in \mathcal{G} \setminus \mathcal{S}^{k}, k\geq K+1$. Therefore, by Lemma \ref{lemma_ineq2} and \eqref{eq_inexact_condition2},  we can derive that
	\begin{equation*}
		\begin{split}
			\|\nu^{k+1}\|_1
			&\leq (\|A^{\top}A\|_2+\beta)\sqrt{n}\|x^{k+1}-z^{k}\|_2+(\frac{\beta\varepsilon}2+n\alpha L_{\underline {c}})\sqrt n \|x^{k+1}-x^{k}\|_2.
		\end{split}
	\end{equation*}
	Applying $\|\nu^{k+1}\|_2\leq\|\nu^{k+1}\|_1$, we get
	\begin{equation}\label{eq_nusubdiffp=1}
		\begin{split}
			\|\nu^{k+1}\|_2
			&\leq 
			(\|A^{\top}A\|_2+\beta+\frac{\beta\varepsilon}2+n\alpha L_{\underline {c}}) \sqrt{n}\|x^{k+1}-x^{k}\|_2+(\|A^{\top}A\|_2+\beta) \sqrt{n}\|x^{k}-x^{k-1}\|_{2}.
		\end{split}
	\end{equation}
	
	Define $\gamma^{k+1}\in\partial\mathcal{H}(x^{k+1},x^{k})$ as  in \eqref{eq_rxry}. By some simple calculations as in \eqref{eq_gamma}, we get
	\begin{equation}
		\label{eq_gammasubdiffp=1}
		\begin{split}
			\|\gamma^{k+1}\|_2 \leq&\left((\|A^{\top}A\|_2+\beta+\frac{\beta\varepsilon}2+n\alpha L_{\underline {c}})\sqrt{n}+\sqrt{2}\beta(1-\varepsilon)\right)\|x^{k+1}-x^{k}\|_2\\
			&+(\|A^{\top}A\|_2+\beta)\sqrt{n}\|x^{k}-x^{k-1}\|_2.
		\end{split}
	\end{equation}
	Summarizing \eqref{eq_impt}, \eqref{eq_gamma} \eqref{eq_nusubdiffp=1} and \eqref{eq_gammasubdiffp=1} completes the proof.
\end{proof}

\begin{lemma}\label{lemma_stationary_point}
	For any $\beta>\|A^{\top}A\|_2$ and the extrapolation parameters $\set{t_{k}}\subset[0,{\bar t})$ satisfying $\sup_{k} t_k<{\bar t}$, 
	let $\set{x^{k}}$ be a sequence generated by FITS$^3$. Then, any cluster point of $\set{x^{k}}$ is a stationary point of $\calE$.
\end{lemma}
\begin{proof}
	By Lemma \ref{lem-bound-sufficient-decrease}(ii), the sequence $\set{x^k}$ is bounded and thus has a cluster point. Take an arbitrary cluster point of $\set{x^k}$
	denoted as $\hat{x}=\lim_{l\rightarrow \infty} x^{k_l}$.  Then, we have $\lim_{l\rightarrow\infty}\calE(x^{k_l})=\calE(\hat{x})$ due to the continuity of $\calE$. By Lemma \ref{lem-relative-error-condition}, one chooses a
	$\nu^{k_l}\in \partial \calE(x^{k_l})$. It holds that $\nu^{k_l}\rightarrow 0$ as $l\rightarrow \infty$ by the inequality \eqref{eq:relative-error-condition} and
	the fact that $\lim_{k \to \infty} \| x^{k} - x^{k-1} \|_{2} = 0$. The closedness of the graph of the limiting subdifferential operator guarantees that $0\in \partial \calE(\hat{x})$.
\end{proof}


Finally, we give our main convergence result of the FITS$^3$ algorithm.
\begin{theorem}\label{thm-global-convergence}
	Assume that $\calH(x,u)$ is a K\L ~function, $\beta>\|A^{\top}A\|_2$ and the extrapolation parameters $\set{t_{k}}\subset[0,{\bar t})$ satisfy $\sup_{k} t_k<{\bar t}$. 
	Let $\set{x^{k}}$ be a sequence generated by FITS$^3$. Then, the sequence $\set{x^{k}}$ converges globally to  a stationary point of $\calE$. Moreover, $\sum_{k=0}^{\infty} \| x^{k} - x^{k-1} \|_{2}< \infty$ holds.
\end{theorem}
\begin{proof}
	We use the abstract convergence theorem in \cite{Ochs2014iPiano} to show the results. Lemma \ref{lem-bound-sufficient-decrease}(i) (ii), Lemma \ref{lem-relative-error-condition} and the continuity of the function  $\calE$  give (H1), (H2) and (H3) required by the abstract convergence theorem  \cite[Theorem 3.7]{Ochs2014iPiano}. Therefore, the sequence $\set{x^{k}}$ converges to an $\overline x$ which is a stationary point of $\calE$ by Lemma \ref{lemma_stationary_point}. Moreover, $\sum_{k=0}^{\infty} \| x^{k} - x^{k-1} \|_{2}< \infty$ holds.
\end{proof}

\begin{remark}
	At least a broad class of objective functions $\calE$ satisfy K\L ~property assumption on $\mathcal{H}$ in Theorem \ref{thm-global-convergence}. If $\calE$ is a semialgebraic function, then $\mathcal{H}$ is also semialgebraic  and thus has K\L ~property.
	
\end{remark}

\begin{remark}
	If we further assume that the function $\varphi\in \Xi_\sigma$ of $\mathcal{H}(x,u)$ in K\L ~property has the form $\varphi(s)=cs^{1-\theta}, \theta\in [0,1)$, we can get the convergence rate of our FITS$^3$ algorithm by a similar analysis with \cite{Attouch2009convergence,articlewen}. To save the space, we omit the details.
\end{remark}

\section{Numerical experiments}
\label{sect_num}
In this section, we show the performance of the  FITS$^3$ algorithm on the group sparse optimization model with a least-squares fidelity. We use the potential functions: $\psi(t)=t^q,0<q<1$.  All the experiments were performed on a personal laptop with Windows 10, 64-bit, 16GB memory, Intel(R) Core(TM) i5-10210U CPU @ 1.60GHz 2.11 GHz, and MATLAB R2016a.

We give a general experiment setting. Unless otherwise specified, we use this setup throughout this section. We set $n=1024$, $m=512$. The ground truth $\underline{x}\in \mathbb{R}^{n}$ has a group structure with the group size $l=16$, that is, the original signal $\underline{x}$ is split equally into $64$ groups. Let $\underline{\mathcal{S}}$ denote the number of nonzero groups, then the sparsity level is $\underline{\mathcal{S}}/64$. The $\underline{\mathcal{S}}$ groups are randomly selected, {whose entries are i.i.d. Gaussian random variables}, and the remaining groups are all set to be zeros. The matrix $B\in\RR^{m\times n}$ is an i.i.d. Gaussian matrix and $A=(orth(B'))'$ in Matlab code.
The noisy simulated observation is
$
b = A \underline{x} + 0.001*\text{randn}(m,1).
$

In the proposed FITS$^3$,
the regularization parameter $\alpha$ is selected by the rule
$
\alpha=5\times 10^{-4}*\alpha_{\text{max}} \; \text{and}\; \alpha_{\text{max}}=\max\{\|A_{\slg_1}^\top b\|_2,  \|A_{\slg_2}^\top b\|_2, .., \|A_{\slg_r}^\top b\|_2\},
$
and
we set $\beta=1.0001$, $\tau=0.2$.
The extrapolation parameter $t_k$ is set as:
\begin{equation}
	\nonumber
	t_k=\frac{a_{k-1}-1}{a_k},\quad a_{k+1}=
	\begin{cases}
		\frac{1+\sqrt{1+4a_k^2}}{2}\quad&\text{if }0\leq k\leq300\\
		a_{301}\quad&\text{if }k>300
	\end{cases},\quad a_{-1}=a_{0}=1.
\end{equation}
The initial point $x^{0}$ is chosen
as the approximate solution of the $\ell_{1}$ regularization, which is solved by ADMM (\cite{Wu2010Augmented}) in very few iterations with random initialization.
The algorithm is stopped, when the following condition is satisfied
\begin{equation*}
	\frac{\| x^{k}-x^{k-1}\|_{2}}{\|x^{k-1}\|_{2}}\leq \text{Tol},
\end{equation*}
with $\text{Tol}=5*10^{-5}$.
We take MAXit=300.
To evaluate the recovery quality, define the following relative error
\begin{equation}
	\nonumber
	\epsilon_{\text{rel}}=\frac{\| x^{\star}-\underline{x}\|_{2}}{\|\underline{x}\|_{2}}.
\end{equation}
The recovered signal $x^{\star}$ is regarded as successful if the relative error less than $0.01$. The success rate is the ratio of success over $50$ trials in this paper.

We mainly test Algorithm~\ref{alg_accelerated_ITS3} with $p=2$ for model~\eqref{eq:object-func_00}, and record the experiment results in various aspects in Subsection~\ref{sect_p2}.
The test on Algorithm~\ref{alg_accelerated_ITS3} with $p=1$ is simply reported in Subsection~\ref{set_p}.

\subsection{Test on Algorithm~\ref{alg_accelerated_ITS3} with $p=2$}\label{sect_p2}

\subsubsection{The convergence and extrapolation effect of FITS$^3$}
\label{sect_performance}
We first test the convergence and extrapolation effect of FITS$^3$ with $p=2$.
For this, we set $q=0.5$,
the number of nonzero groups $\underline{\mathcal{S}}=12$, i.e., the sparsity level is $18.75\%$. The recovery result is shown in Fig. \ref{fig:singal_recovery0} (a) and the recovery error $\left(\frac{\| x^{k}-\underline{x}\|_{2}}{\|\underline{x}\|_{2}}\right)$ is  shown in Fig. \ref{fig:singal_recovery0} (b).
It can be seen that FITS$^3$ performs well on the sparse signal recovery problem.
\begin{figure}[htbp]
	\centering
	\begin{tabular}{c@{\hspace{1mm}}c@{\hspace{1mm}}c@{\hspace{1mm}}}
		\subfloat[]
		{\includegraphics[width=0.32\textwidth]{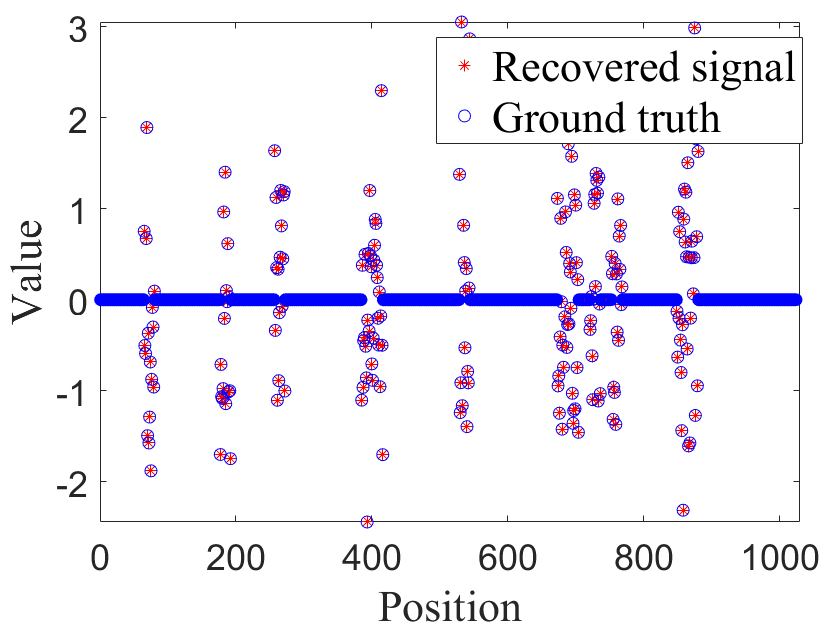}}&
		\subfloat[]
		{\includegraphics[width=0.32\textwidth]{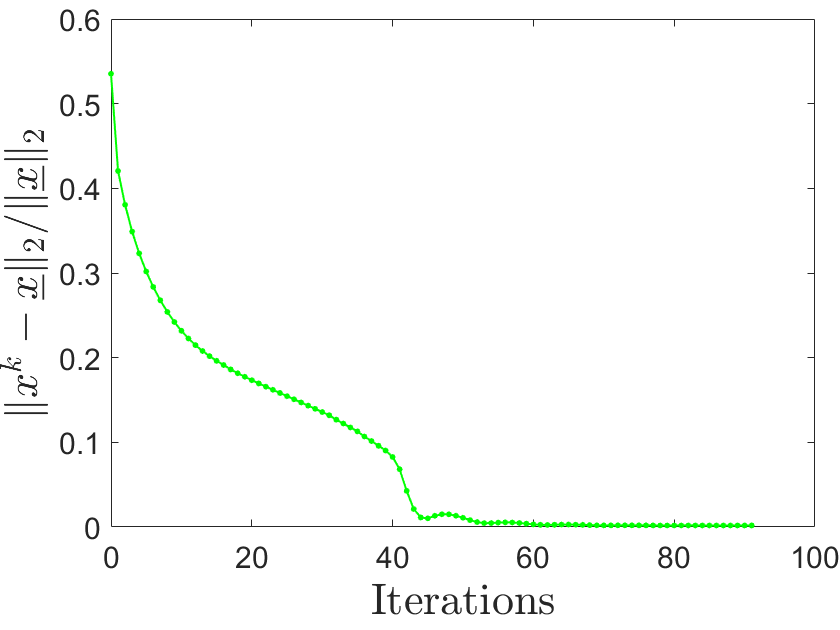}}&
		\subfloat[]
		{\includegraphics[width=0.32\textwidth]{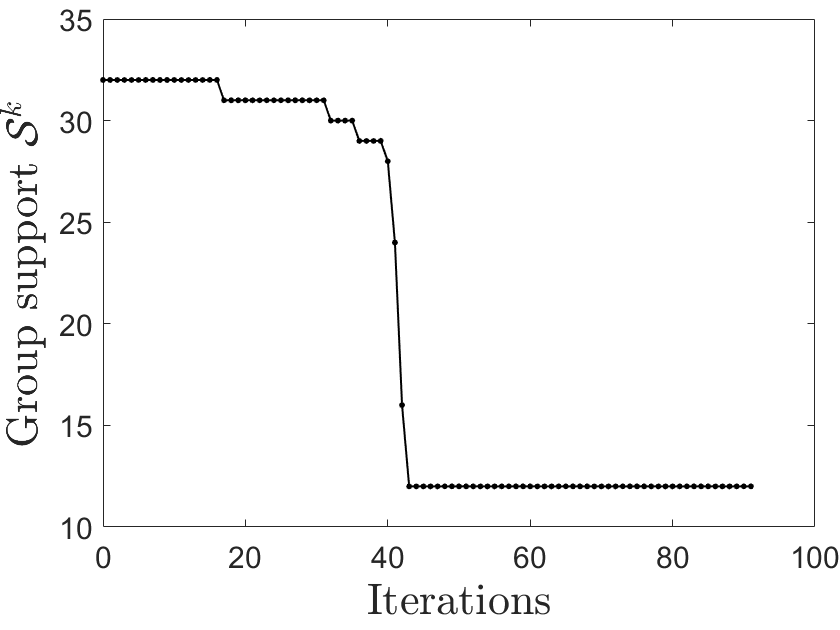}}\\
		\subfloat[]
		{\includegraphics[width=0.32\textwidth]{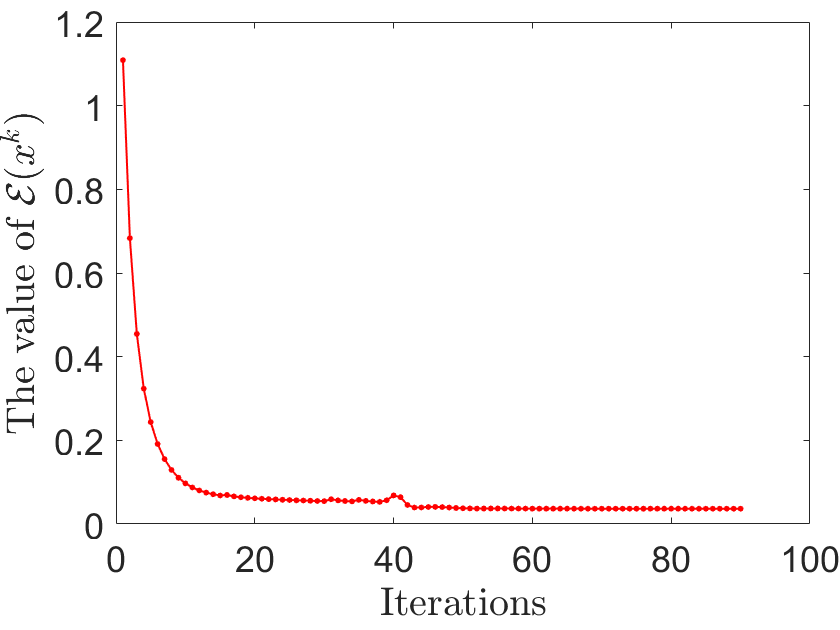}}&
		\subfloat[]
		{\includegraphics[width=0.32\textwidth]{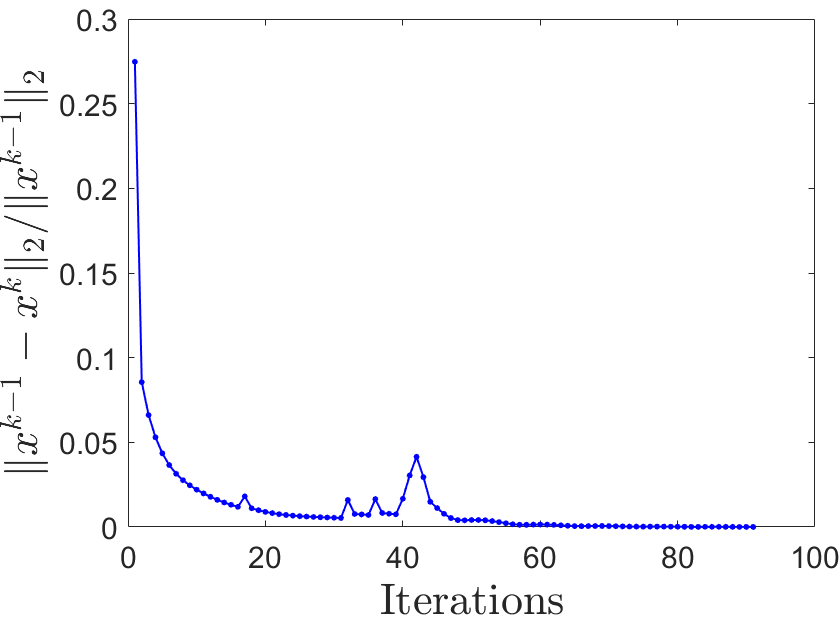}}&
		\subfloat[]
		{\includegraphics[width=0.32\textwidth]{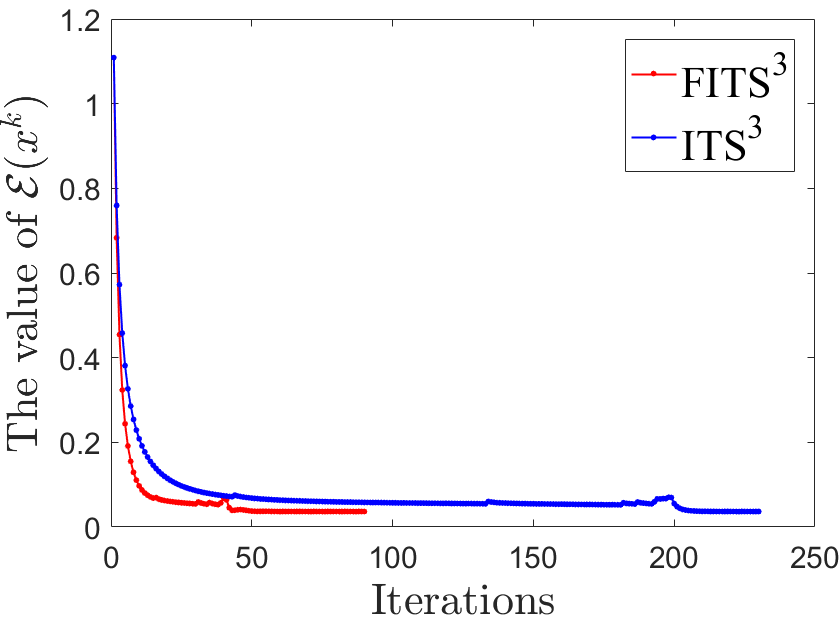}}
		\\		
	\end{tabular}
	\caption{\footnotesize Test on FITS$^3$ with $p=2$. (a) The recovery by FITS$^3$. (b)(c)(d)(e) The convergence justification of FITS$^3$. (f) The speed up effect of  extrapolation.}
	\label{fig:singal_recovery0}
\end{figure}

Next, we show the convergence behavior of the FITS$^3$ algorithm. 
Experimental results are reported in Fig. \ref{fig:singal_recovery0} (b)-(e). The iteration points become closer and closer to the ground truth as shown in Fig \ref{fig:singal_recovery0} (b). The group support set shown in Fig. \ref{fig:singal_recovery0} (c) decreases monotonically and converges, which is consistent with the conclusion in \eqref{eq_decrease_mont}. The sequences of $\calE(x^{k})$ and the relative error $\left(\frac{\| x^{k-1}-x^{k}\|_{2}}{\|x^{k-1}\|_{2}}\right)$  shown in  Fig. \ref{fig:singal_recovery0} (d)(e) converge, although oscilate slightly, which are consistent with our theoretical results. 

We show the speed up effect of  extrapolation. When the extrapolation parameter $t=0$, the FITS$^3$ algorithm degenerates to the  algorithm without extrapolation (ITS$^3$). We run ITS$^3$  and  FITS$^3$ under the same settings. The sequence of objective function value $\calE(x^k)$ of FITS$^3$ and ITS$^3$ are shown in Fig. \ref{fig:singal_recovery0} (f). It can be seen that the value of $\calE(x^k)$ of FITS$^3$ decreases and converges faster and FITS$^3$ requires fewer iteration steps to reach the stopping condition. 


\subsubsection{Choice of $q$}\label{sect_q}
In this subsection, we discuss the choice of $q$. We test the FITS$^3$ algorithm with different $q\in\set{0.1, 0.3, 0.5, 0.7, 0.9}$. The success rates of FITS$^3$ with different $q$ are reported in Fig. \ref{fig:q}(a). It shows that under the general setting,
the success rates of FITS$^3$ with $q=0.3$ and $q=0.5$ are higher than others.
In the following experiments, we take $q=0.5$.
\begin{figure}[htbp]
	\centering
	\begin{tabular}{c@{\hspace{0mm}}c@{\hspace{0mm}}c@{\hspace{0mm}}c@{\hspace{0mm}}}
		\subfloat[]
		{\includegraphics[width=0.24\textwidth]{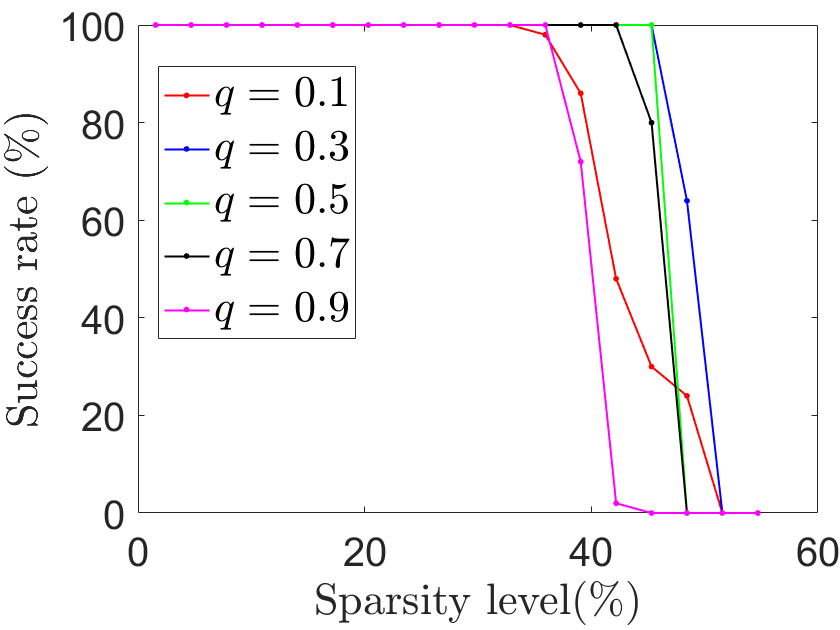}}&
		\subfloat[]
		{\includegraphics[width=0.24\textwidth]{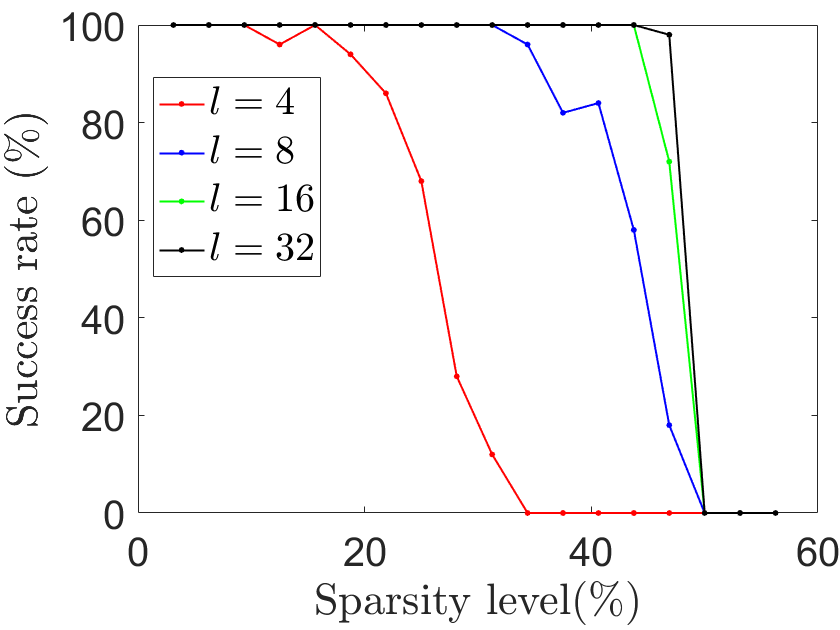}}&
		\subfloat[]
		{\includegraphics[width=0.24\textwidth]{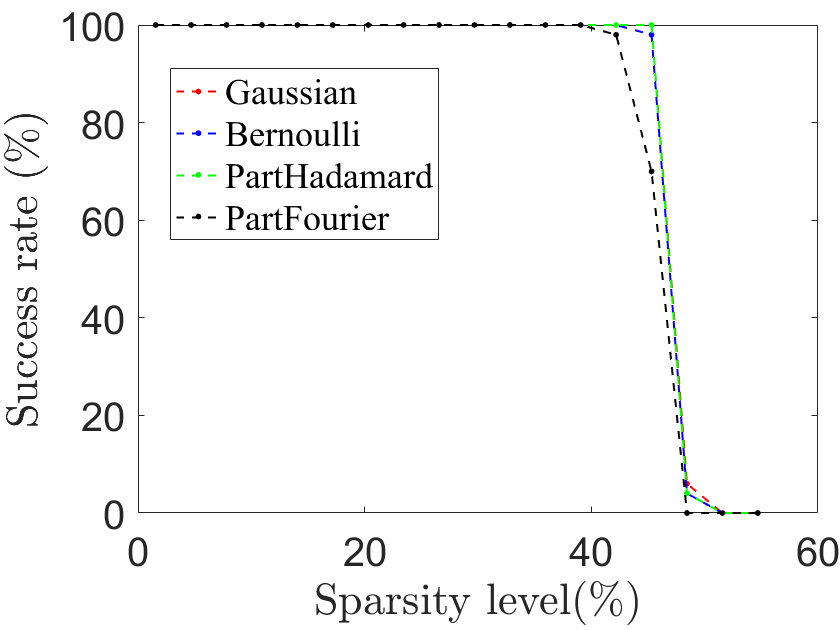}}&	
		\subfloat[]
		{\includegraphics[width=0.24\textwidth]{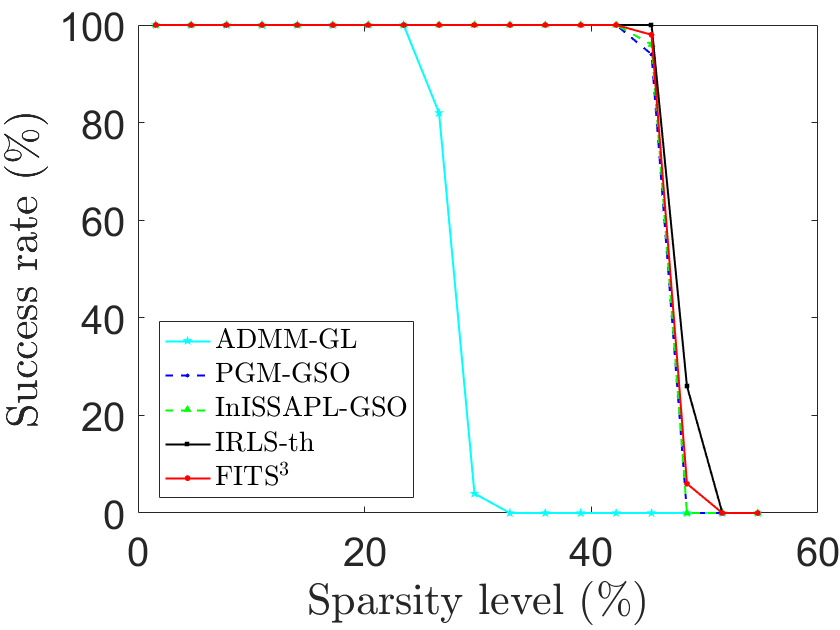}}\\
	\end{tabular}
	\caption{\footnotesize Success rate tests and comparisons on FITS$^3$ with $p=2$. (a)(b)(c) The success rates of FITS$^3$ for group sparse minimization models with different $q$, different group sizes and different types of measurement matrix $A$, respectively.  (d) Comparison on success rates between non-Lipschitz group sparse  recovery algorithms and ADMM-GL algorithm.}
	\label{fig:q}
\end{figure}

\subsubsection{Adaptability to different group sizes and different types of $A$}\label{set_groupsize}
In this subsection, we test the sensitivity of FITS$^3$  with $q=0.5$ on group sizes and types of $A$ in terms of success rate.
We first show the adaptability of FITS$^3$ on group sizes.
We test FITS$^3$ for four problems with different group sizes $l\in\set{4,8,16,32}$. 
The success rates are shown in Fig.  \ref{fig:q} (b). It can be seen that the larger the group size, the higher the success rate.

We next test the sensitivity of FITS$^3$ on different measurement matrices $A$. We use matrix $A$ generated from $A=(orth(B'))'$ in Matlab, where $B$ is a random Gaussian
matrix, random Bernoulli matrix, random PartHadamard matrix or random PartFourier matrix. 
Success rates of recoveries by FITS$^3$ are presented in Fig. \ref{fig:q} (c). It can be seen that FITS$^3$ performs consistently well in problems with different measurement matrices. 
%
%

\subsubsection{Comparisons}\label{set_comp}
In this subsection, we compare the performance of FITS$^3$ with some existing algorithms solving non-Lipschitz $\ell_{p,q}$ regularization models, including 
PGM-GSO (\cite{hu2017group}), InISSAPL-GSO (\cite{xue2019efficient}) and IRLS-th (\cite{feng2020l2}). We also compare it to the ADMM-GL (\cite{boyd2011distributed}) solving the convex  group lasso model ($\ell_2$-$\ell_{2,1}$ model) (\cite{yuan2006model}), which is a baseline algorithm for group sparse recovery. The PGM-GSO, InISSAPL-GSO, IRLS-th and our FITS$^3$ are used to solve the non-Lipschitz $\ell_2$-$\ell_{2,q}$ minimization model.
The parameters and stopping conditions of the compared algorithms follow their default settings. We set $q=0.5$ in the PGM-GSO, InISSAPL-GSO and IRLS-th since these algorithms give their best results at $q=0.5$.  We compare these algorithms in terms of the success rate, recovery accuracy and CPU cost.

Fig. \ref{fig:q} (d) shows the comparison on the success rates of these algorithms under the general setting for recovery problems with different sparsity levels. It can be observed that the recovery success rates of FITS$^3$, PGM-GSO, InISSAPL-GSO, IRLS-th are very close and all higher than ADMM-GL.

In Table~\ref{tab_compp1}, we compare the CPU costs of these algorithms for recovery problems with different scales and sparsity levels. In particular, we set $n\in\set{1024,4096,8192,12288,16384}$, $m=n/2$, the group size $l=16$ and the sparsity levels $5\%$, $10\%$, $15\%$ and $20\%$. It can be seen that non-Lipschitz methods PGM-GSO, InISSAPL-GSO, IRLS-th and FITS$^3$  reach similar recovery accuracies (better than  the convex ADMM-GL method), while the proposed FITS$^3$ usually takes the least CPU cost.
This efficiency advantage becomes more and more greater as the problem scale $n$ increases.
\subsection{Test on Algorithm~\ref{alg_accelerated_ITS3} with $p=1$
}\label{set_p}
Here we test the performance of the FITS$^3$ algorithm for signal recovery problems with intra-group sparsity structures.
We set $n=1024$, $m=512$, the group size $l=32$, and the number of nonzero group $\underline{\mathcal{S}}=16$.
To simulate the intra-group sparsity structures, we denote the number of nonzero elements of each nonzero group as $\underline{s}$, and consider $\underline{s}=6$ and 9.
Fig. \ref{fig:singal_recoveryp1p2} shows the recovery results of FITS$^3$ with $p=1$, $q=0.5$ and $p= 2$, $q=0.5$.
We  observe that for such signal recovery problems with intra-group
sparsity structures, Algorithm~\ref{alg_accelerated_ITS3} with $p=1$ works quite well and much better than Algorithm~\ref{alg_accelerated_ITS3} with $p=2$.



\begin{figure}[htbp]
	\centering
	\begin{tabular}{c@{\hspace{1mm}}c@{\hspace{1mm}}c@{\hspace{1mm}}c@{\hspace{1mm}}}
		\subfloat[{\tiny $p=1, \underline{s}=6: \epsilon_{\text{rel}}=0.0025$}]
		{\includegraphics[width=0.24\textwidth]{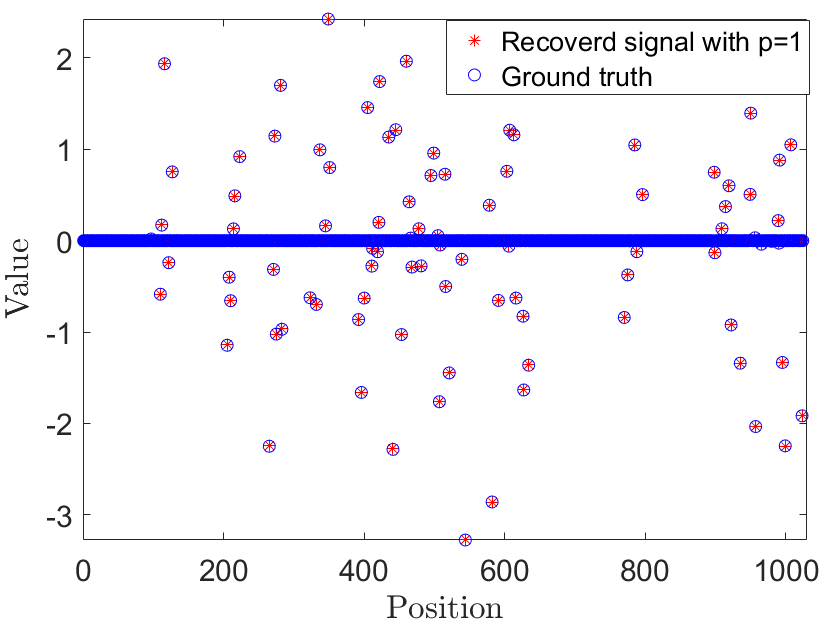}}&
		\subfloat[{\tiny $p=2, \underline{s}=6: \epsilon_{\text{rel}}=0.1212$}]
		{\includegraphics[width=0.24\textwidth]{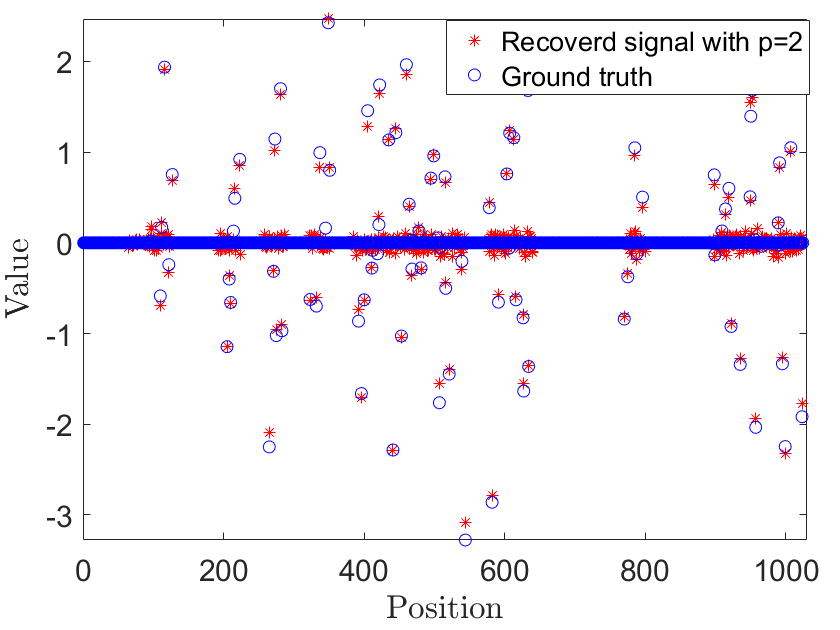}}		
		\subfloat[{\tiny $p=1, \underline{s}=9 : \epsilon_{\text{rel}}=0.0030$}]
		{\includegraphics[width=0.24\textwidth]{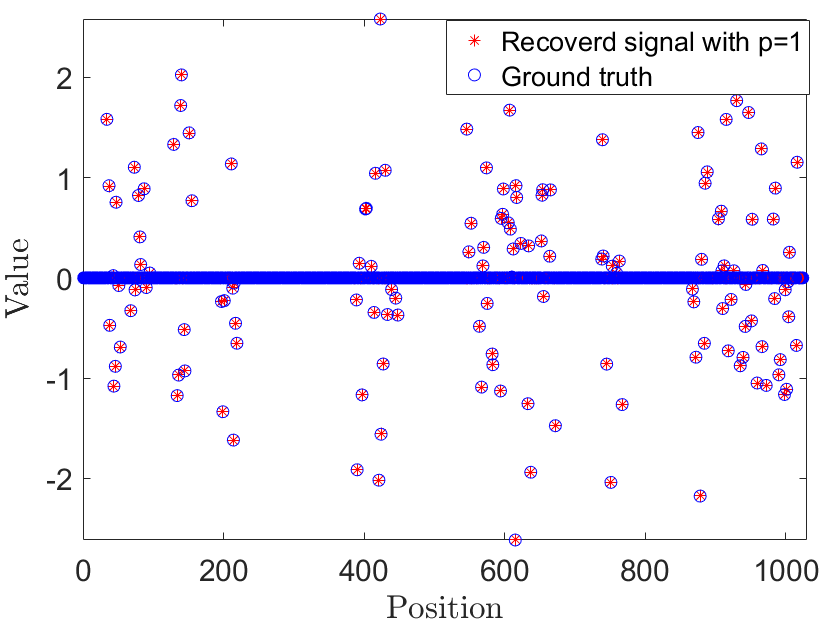}} &
		\subfloat[{\tiny $p=2, \underline{s}=9: \epsilon_{\text{rel}}=0.1525$}]
		{\includegraphics[width=0.24\textwidth]{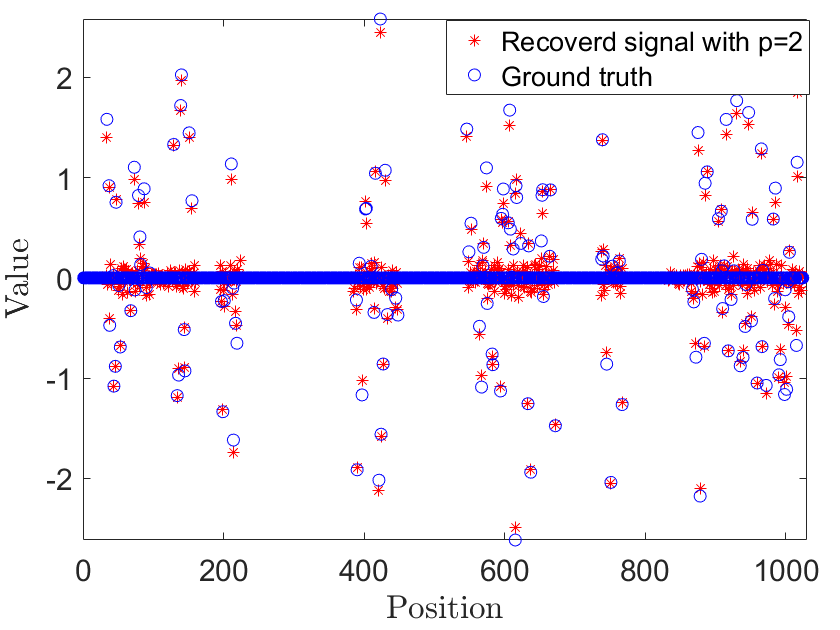}}
	\end{tabular}
	\caption{\footnotesize Signal recovery by FITS$^3$ with $p=1$ and $p=2$ for the signal recovery problems with intra-group sparsity structures. (a)(b) The intra-group sparsity level $\underline{s}/l=6/32$. (c)(d) The intra-group sparsity level $\underline{s}/l=9/32$.}
	\label{fig:singal_recoveryp1p2}
\end{figure}
\section{Conclusions}\label{sec:conclu}
In this paper, we proposed a new simple algorithm named as FITS$^3$ for a general non-convex non-Lipschitz group sparse regularization model with a least-squares fidelity, by integrating thresholding operation, group support-and-scale shrinkage, linearization and  extrapolation techniques.
This iterative algorithm does not need to solve any linear or nonlinear system, and in two most important cases $p=1$ and $p=2$ it contains only simple operations like matrix-vector multiplication and soft thresholding. It thus  is especially suitable for solving median- or large-scale problems. 
We also established the sequence  convergence  property of the proposed FITS$^3$ under an inexact condition at each iteration, which is applicable to our FITS$^3$ with a general $p\in[1,+\infty).$
Numerical experiments illustrated the advantages of the FITS$^3$, showing great potentials in various group sparse recovery applications in the current era of big data.

\section*{Acknowledgments}

This work was partially supported by National Natural Science Founation of China (Grants 12271273, 11871035), the Key Program (21JCZDJC00220) of Natural Science Foudation
of Tianjin, China.

\begin{landscape}

\begin{table}[tbhp]
	\small
	\caption{\footnotesize Comparisons on relative errors (abbreviated as $\epsilon_{\text{rel}}$) and CPU time (seconds)  between non-Lipschitz group sparse  recovery algorithms and ADMM-GL algorithm for problems with different scales $n\in\set{1024,4096,8192,12288,16384}$. }
	\label{tab_compp1}
	\begin{center}
		\begin{tabular}{|c|c|c|c|c|c|c|c|c|c|c|}
			\hline
			\multicolumn{11}{|c|}{Sparsity level: $5\%$}\\
			\hline
			\multirow{2}{*}{Method}&\multicolumn{2}{c}{$1024$}&\multicolumn{2}{|c|}{$4096$}&\multicolumn{2}{c|}{$8192$}&\multicolumn{2}{c|}{$12288$}&\multicolumn{2}{c|}{$16384$}\\
			\cline{2-11}  &$\epsilon_{\text{rel}}$&time &$\epsilon_{\text{rel}}$&time&$\epsilon_{\text{rel}}$&time&$\epsilon_{\text{rel}}$&time&$\epsilon_{\text{rel}}$&time\\
			\hline
			ADMM-GL     &0.0023 &0.14&0.0022 &2.85  &0.0023 &11.23 &0.0023 &27.93&0.0022 &56.83\\
			PGM-GSO         &0.0013 &0.29&0.0015 &6.25  &0.0014 &25.78 &0.0015 &64.81&0.0013 &132.94\\
			InISSAPL-GSO    &0.0013 &0.06&0.0015 &1.19  &0.0014 &7.49  &0.0015 &18.63&0.0013 &42.67\\
			IRLS-th         &0.0013 &0.06&0.0015 &1.24  &0.0014 &6.89  &0.0015 &21.23&0.0013 &50.74\\
			FITS$^3$   &0.0013 &\bf{0.05}&0.0015 &\bf{0.79}  &0.0014 &\bf{3.95}  &0.0015 &\bf{10.81}&0.0014 &\bf{24.46}\\
			\hline
			\hline
			\multicolumn{11}{|c|}{Sparsity level: $10\%$}\\
			\hline
			\multirow{2}{*}{Method}&\multicolumn{2}{c}{$1024$}&\multicolumn{2}{|c|}{$4096$}&\multicolumn{2}{c|}{$8192$}&\multicolumn{2}{c|}{$12288$}&\multicolumn{2}{c|}{$16384$}\\
			\cline{2-11}  &$\epsilon_{\text{rel}}$&time &$\epsilon_{\text{rel}}$&time&$\epsilon_{\text{rel}}$&time&$\epsilon_{\text{rel}}$&time&$\epsilon_{\text{rel}}$&time\\
			\hline
			ADMM-GL      &0.0023 &0.20&0.0026 &3.98  &0.0026 &15.25 &0.0024 &37.35&0.0026 &65.70\\
			PGM-GSO         &0.0015 &0.34&0.0016 &7.01  &0.0015 &28.57 &0.0015 &68.16&0.0015 &128.31\\
			InISSAPL-GSO    &0.0015 &0.08&0.0016 &2.66  &0.0015 &14.35 &0.0015 &30.29&0.0015 &64.90\\
			IRLS-th         &0.0015 &0.06&0.0016 &1.41  &0.0016 &8.15  &0.0015 &22.22&0.0015 &48.98\\
			FITS$^3$   &0.0015 &\bf{0.05}&0.0016 &\bf{0.93}  &0.0016 &\bf{4.78}  &0.0015 &\bf{12.28}&0.0015 &\bf{25.35}\\
			\hline
			\hline
			\multicolumn{11}{|c|}{Sparsity level: $15\%$}\\
			\hline
			\multirow{2}{*}{Method}&\multicolumn{2}{c}{$1024$}&\multicolumn{2}{|c|}{$4096$}&\multicolumn{2}{c|}{$8192$}&\multicolumn{2}{c|}{$12288$}&\multicolumn{2}{c|}{$16384$}\\
			\cline{2-11}  &$\epsilon_{\text{rel}}$&time &$\epsilon_{\text{rel}}$&time&$\epsilon_{\text{rel}}$&time&$\epsilon_{\text{rel}}$&time&$\epsilon_{\text{rel}}$&time\\
			\hline
			ADMM-GL      &0.0025 &0.23&0.0030 &4.86  &0.0027 &19.82 &0.0028 &47.90&0.0027 &84.27\\
			PGM-GSO         &0.0015 &0.37&0.0015 &6.79  &0.0016 &28.54&0.0015 &66.78&0.0015 &127.97\\
			InISSAPL-GSO    &0.0015 &0.10&0.0015 &7.22  &0.0016 &31.69 &0.0015 &75.88&0.0015 &161.81\\
			IRLS-th         &0.0015 &\bf{0.06}&0.0015 &1.49 &0.0016 &8.51 &0.0015 &24.01&0.0015 &51.25\\
			FITS$^3$   &0.0016 &0.07&0.0015 &\bf{1.06}  &0.0016 &\bf{5.02} &0.0015 &\bf{13.47}&0.0015 &\bf{28.39}\\
			\hline
			\hline
			\multicolumn{11}{|c|}{Sparsity level: $20\%$}\\
			\hline
			\multirow{2}{*}{Method}&\multicolumn{2}{c}{$1024$}&\multicolumn{2}{|c|}{$4096$}&\multicolumn{2}{c|}{$8192$}&\multicolumn{2}{c|}{$12288$}&\multicolumn{2}{c|}{$16384$}\\
			\cline{2-11}  &$\epsilon_{\text{rel}}$&time &$\epsilon_{\text{rel}}$&time&$\epsilon_{\text{rel}}$&time&$\epsilon_{\text{rel}}$&time&$\epsilon_{\text{rel}}$&time\\
			\hline
			ADMM-GL      &0.0028 &0.35&0.0037 &6.61  &0.0032 &24.80 &0.0036 &64.39&0.0034 &108.11\\
			PGM-GSO         &0.0018 &0.32&0.0016 &6.87  &0.0017 &28.47&0.0016 &67.58&0.0016 &125.96\\
			InISSAPL-GSO    &0.0018 &0.43&0.0016 &14.09  &0.0017 &62.12 &0.0016 &156.86&0.0016 &340.62\\
			IRLS-th         &0.0018 &\bf{0.08}&0.0016 &2.04  &0.0017 &10.13 &0.0016 &32.71&0.0016 &78.16\\
			FITS$^3$   &0.0018 &0.09&0.0017 &\bf{1.43}  &0.0017 &\bf{6.13} &0.0017 &\bf{15.80}&0.0016 &\bf{33.10}\\
			\hline
		\end{tabular}
	\end{center}
\end{table}
\end{landscape}

\vskip2mm



          \end{document}